\documentclass[final,leqno,showlabe]{siamltex}
\usepackage{amsmath}
\usepackage{stmaryrd}
\usepackage{booktabs}
\usepackage{cite}
\usepackage{amssymb}
\usepackage{verbatim}
\usepackage[all,cmtip]{xy}
\usepackage[top=3.54cm,bottom=4.54cm,left=2.94cm,right=2.94cm ]{geometry}

\input psfig.sty

\renewcommand{\theequation}{\arabic{section}.\arabic{equation}}
\renewcommand{\thetable}{\arabic{section}.\arabic{table}}
\renewcommand{\thefigure}{\arabic{section}.\arabic{figure}}
\newcommand{\bx}{{\bf x} }

\newcommand{\bee}{{\bf e} }
\newcommand{\bez}{{\bf z}}

\newcommand{\calF}{{\mathcal{F}} }
\newcommand{\p}{\partial}
\newcommand{\eps}{\varepsilon}
\newcommand{\sinc}{{\text{sinc}}}

\newtheorem{rmk}{Remark}[section]
\newcommand{\be}{\begin{equation}}
\newcommand{\ee}{\end{equation}}
\newcommand{\ba}{\begin{array}}
\newcommand{\ea}{\end{array}}
\newcommand{\bea}{\begin{eqnarray}}
\newcommand{\eea}{\end{eqnarray}}
\newcommand{\beas}{\begin{eqnarray*}}
\newcommand{\eeas}{\end{eqnarray*}}

\title{A uniformly accurate (UA) multiscale time integrator pseudospectral method  for the Dirac equation in the nonrelativistic limit regime\thanks{This work was partially support by  the Ministry of Education of Singapore grant R-146-000-196-112 (W. Bao and X. Jia), by NSF grants DMS-1217066 and DMS-1419053 (Y. Cai)
and by the ANR project BECASIM ANR-12-MONU-0007-02 (Q. Tang).}}
\author{Weizhu Bao\thanks{Department of Mathematics, National University of Singapore, Singapore
119076 ({\tt matbaowz@nus.edu.sg},
URL: http://www.math.nus.edu.sg/\~{}bao/)}
\and Yongyong Cai\thanks{Beijing Computational Science Research Center,
   Beijing 100094, China; and Department of Mathematics,
Purdue University, West Lafayette, IN 47907, USA ({\tt yongyong.cai@gmail.com})}
\and Xiaowei Jia\thanks{Department of Mathematics,  National University of
Singapore, Singapore 119076 ({\tt A0068124@nus.edu.sg})}
\and qinglin Tang\thanks{Institut Elie Cartan de Lorraine, Universit\'e de Lorraine, Inria Nancy-Grand Est,
F-54506 Vandoeuvre-l\`es-Nancy Cedex, France
({\tt tqltql2010@gmail.com})}}
\date{}
\begin{document}

\maketitle

\begin{abstract}
We propose and rigourously analyze  a multiscale time integrator Fourier
pseudospectral (MTI-FP) method  for the (linear) Dirac equation
with a  dimensionless parameter $\eps\in(0,1]$ which is inversely proportional to the speed of light.
In the nonrelativistic limit regime, i.e. $0<\eps\ll 1$, the solution exhibits
highly oscillatory  propagating waves
with wavelength $O(\varepsilon^2)$ and $O(1)$ in time and space, respectively.
 Due to the rapid temporal oscillation, it is quite challenging in
designing and analyzing numerical methods with  uniform error bounds in $\eps\in(0,1]$.
We present the  MTI-FP method based on properly adopting a
multiscale decomposition of the solution of the Dirac equation
and applying the exponential wave integrator with appropriate numerical quadratures.
By a careful study of the error propagation and using the energy method, we
establish two independent error estimates via two different mathematical
approaches as $h^{m_0}+\frac{\tau^2}{\eps^2}$ and $h^{m_0}+\tau^2+\eps^2$,
where $h$ is the mesh size, $\tau$ is  the time step and $m_0$ depends on the regularity
of the solution. These two error bounds immediately imply that the MTI-FP method converges uniformly and optimally
in space with exponential convergence rate if the solution is smooth, and uniformly
in time with linear convergence rate at $O(\tau)$ for all $\eps\in(0,1]$ and optimally
with quadratic convergence rate at  $O(\tau^2)$ in the regimes when either $\eps=O(1)$
or $0<\eps\lesssim \tau$. Numerical results are reported to demonstrate that our error estimates
are optimal and sharp. Finally, the MTI-FP method is applied to study numerically the convergence
rates of the solution of the Dirac equation to those of its limiting models when $\eps\to0^+$.
\end{abstract}



\section{Introduction}\setcounter{equation}{0}
Quantum mechanics and relativity theory were the two major
physics discoveries in the last century. The first successful
attempt to consistently integrate these two fundamental theories was
made by the British physicist Paul Dirac in 1928 \cite{Dirac1,Dirac2}, resulting in the equation
known as the Dirac equation. It describes  the evolution of
spin-$1/2$ massive particles, such as electrons and quarks.
It is a relativistic version of the Schr\"odinger equation for quantum
mechanics, consistent with Albert Einstein's special relativity.
Dirac's theory led to the rigorous explanation for the fine structure of the hydrogen spectrum
and the prediction of the antimatter \cite{Ad}, and predated the experimental discovery of positron.
In different parameter limits, the Dirac equation collapses to the Pauli equation,
the Schr\"{o}dinger equation, and the Weyl equation, respectively.
Since the first experimental realization of graphene in 2003
\cite{AMPGMKWTNLG,NGMJZDGF},
much attention has been drawn to the study of the structures and/or  dynamical
properties of graphene and graphite as well as two dimensional (2D) materials \cite{NGPNG},
in which the Dirac equation plays an important role. This remarkable experimental advance renewed the
interest on the theoretical analysis and numerical simulation
of the Dirac equation and/or the (nonlinear) Schr\"{o}dinger equation without/with
external potentials, especially the honeycomb lattice potential \cite{AZ, FW,FW2}.

After proper nondimensionlization and dimension reduction,
the (linear) Dirac equation for the spin-$1/2$ particles
with external electromagnetic potential in $d$ ($d=3,2,1$) dimensions  reads \cite{Ad,BCJ,BHM,Dirac1,Dirac2,Sch,NSG,Hun}:
\be
\label{SDEd}
i\partial_t\Psi(t,\bx)=\left[-\frac{i}{\eps}\sum_{j=1}^{d}\alpha_j
\partial_j+\frac{1}{\eps^2}\beta+V(t,\bx)I_4-\sum_{j=1}^{d}A_j(t,\bx)\alpha_j\right]
\Psi(t,\bx), \quad \bx\in{\mathbb R}^d,
\ee
with the initial data
\be\label{SDEd-init}
\Psi(t=0,\bx)=\Psi_0(\bx),\qquad \bx\in{\mathbb R}^d,
\ee
where $i=\sqrt{-1}$,  $t$ is time, $\bx=(x_1,\ldots,x_d)^T\in {\mathbb R}^d$
is the spatial coordinate vector, $\partial_k=\frac{\partial}{\partial x_k}$ ($k=1,\ldots,d$),
 $\Psi :=\Psi(t,\bx)=(\psi_1(t,\bx),\psi_2(t,\bx), \psi_3(t,\bx), \psi_4(t,\bx))^T\in\mathbb{C}^4$
 is the complex-valued vector
wave function of the ``spinorfield'', $\eps\in(0,1]$ is a dimensionless parameter
inversely proportional to the speed of light.
 $I_n$ is the $n\times n$ identity matrix for $n\in {\mathbb N}$,
$V:=V(t,\bx)$ is the real-valued electrical potential and
${\bf A}:={\bf A}(t,\bx)=(A_1(t,\bx), \ldots A_d(t,\bx))^T$ is the real-valued magnetic potential vector.
In addition, the $4\times 4$
matrices $\alpha_1$, $\alpha_2$, $\alpha_3$ and $\beta$ are defined as
\be \label{alpha}
\alpha_1=\left(\begin{array}{cc}
\mathbf{0} & \sigma_1  \\
\sigma_1 & \mathbf{0}  \\
\end{array}
\right),\qquad
\alpha_2=\left(\begin{array}{cc}
\mathbf{0} & \sigma_2 \\
\sigma_2 & \mathbf{0} \\
\end{array}
\right), \qquad
\alpha_3=\left(\begin{array}{cc}
\mathbf{0} & \sigma_3 \\
\sigma_3 & \mathbf{0} \\
\end{array}
\right),\qquad
\beta=\left(\begin{array}{cc}
I_{2}& \mathbf{0} \\
\mathbf{0} & -I_{2} \\
\end{array}
\right),
\ee
 where $\sigma_1$, $\sigma_2$,  $\sigma_3$
are the Pauli matrices  given by
\be\label{Paulim}
\sigma_{1}=\left(
\begin{array}{cc}
0 & 1  \\
1 & 0  \\
\end{array}
\right), \qquad
\sigma_{2}=\left(
\begin{array}{cc}
0 & -i \\
i & 0 \\
\end{array}
\right),\qquad
\sigma_{3}=\left(
\begin{array}{cc}
1 & 0 \\
0 & -1 \\
\end{array}
\right).
\ee

The Dirac equation \eqref{SDEd} with (\ref{SDEd-init}) conserves the total mass
\be\label{norm}
\|\Psi(t,\cdot)\|^2:=\int_{{\mathbb R}^d}|\Psi(t,\bx)|^2\,d\bx=
\int_{{\mathbb R}^d}\sum_{j=1}^4|\psi_j(t,\bx)|^2\,d\bx\equiv \|\Psi(0,\cdot)\|^2
=\|\Psi_0\|^2, \qquad t\ge0.
\ee
In addition, if the external electromagnetic potentials are time independent,
i.e. $V(t,\bx)=V(\bx)$ and $A_j(t,\bx)=A_j(\bx)$ ($1\leq j\leq d$),
the Dirac equation \eqref{SDEd} with (\ref{SDEd-init}) conserves the energy
\be\label{engery60}
E(t):=\int_{\mathbb{R}^d}\left(-\frac{i}{\varepsilon}
\sum_{j=1}^d\Psi^*\alpha_j\partial_j\Psi+\frac{1}{\varepsilon^2}\Psi^*\beta\Psi+V(\bx)|\Psi|^2-
\sum_{j=1}^dA_j(\bx)\Psi^*\alpha_j\Psi\right)d\bx\equiv E(0),\quad t\ge0,
\ee
where $\Psi^*=\overline{\Psi^T}$ denotes the conjugate transpose of  $\Psi$.
In the nonrelativistic limit regime, i.e. $0<\eps\ll 1$, as proven in
\cite{White,Sch,NM,N,Mau,Hun,CC,BMP}, the solution
$\Psi$ of the Dirac equation \eqref{SDEd} can be split into the electron part and the positron part, i.e.,
\be\label{eq:MD:1st}
\Psi=e^{-it/\eps^2}\begin{pmatrix}\varphi_1\\ \varphi_2 \\0\\0\end{pmatrix}+
e^{it/\eps^2}\begin{pmatrix}0\\ 0 \\ \varphi_3\\ \varphi_4\end{pmatrix}+O(\eps)=
e^{-it/\eps^2}\Phi_e+e^{it/\eps^2}\Phi_p+O(\eps),
\ee
where both the `electron component' $\Phi_e$ and `positron component'
$\Phi_p$ satisfy the (different) Schr\"odinger equation \cite{Sch,BMP}.
In addition, a higher order $O(\eps^2)$ approximate model of the
Pauli-type equation was provided in \cite{NM,Mau}.
For details of the nonrelativistic limit of the Dirac equation
\eqref{SDEd}, we refer the readers to
\cite{Sch,BMP,NM,Mau} and references therein.

In practice, for lower dimensions $d=1,2$, the Dirac equation \eqref{SDEd}
can be split into two equivalent sets of decoupled equations with two components each \cite{BCJ}
and thus can be reduced to the following equation for
$\Phi:=\Phi(t,\bx)=(\phi_1(t,\bx),\phi_2(t,\bx))^T\in \Bbb C^2$ as
\be
\label{SDEdd}
i\partial_t\Phi(t,\bx)=\left[-\frac{i}{\eps}\sum_{j=1}^{d}\sigma_j
\partial_j+\frac{1}{\eps^2}\sigma_3+V(t,\bx)I_2-\sum_{j=1}^{d}A_j(t,\bx)\sigma_j\right]
\Phi(t,\bx), \quad \bx\in{\mathbb R}^d,
\ee
with the initial data
\be\label{SDEdd-init}
\Phi(t=0,\bx)=\Phi_0(\bx),\qquad \bx\in{\mathbb R}^d,\qquad d=1,2,
\ee
where   $\Phi=(\psi_1,\psi_4)^T$ (or $\Phi=(\psi_2,\psi_3)^T$  in one dimension (1D) and under the transformation
$x_2\to -x_2$ and $A_2\to -A_2$ in 2D).  As a result of its simplicity compared to \eqref{SDEd},
the Dirac equation \eqref{SDEdd} has been widely used when considering the
1D and 2D cases \cite{WT,XST,BCJ}.

Similarly, the Dirac equation \eqref{SDEdd} with (\ref{SDEdd-init}) conserves the total mass
\be\label{norm:1}
\|\Phi(t,\cdot)\|^2:=\int_{{\mathbb R}^d}|\Phi(t,\bx)|^2\,d\bx=
\int_{{\mathbb R}^d}\sum_{j=1}^2|\phi_j(t,\bx)|^2\,d\bx\equiv \|\Phi(0,\cdot)\|^2
=\|\Phi_0\|^2, \qquad t\ge0.
\ee
Furthermore, if the external electromagnetic potentials are time-independent,
i.e. $V(t,\bx)=V(\bx)$ and $A_j(t,\bx)=A_j(\bx)$ ($1\leq j\leq d$),
the Dirac equation \eqref{SDEdd} with (\ref{SDEdd-init})  conserves the energy
\be\label{engery601}
E(t):=\int_{\mathbb{R}^d}\left(-\frac{i}{\varepsilon}
\sum_{j=1}^d\Phi^*\sigma_j\partial_j\Phi+\frac{1}{\varepsilon^2}\Phi^*\sigma_3\Phi+V(\bx)|\Phi|^2-
\sum_{j=1}^dA_j(\bx)\Phi^*\sigma_j\Phi\right)d\bx\equiv E(0),\; t\ge0.
\ee
For the Dirac equation \eqref{SDEdd}, one can obtain  the nonrelativistic limit  which is similar
to  \eqref{eq:MD:1st} and the detail is omitted here for brevity \cite{Sch,BMP,NM,Mau}.

There have been extensive theoretical and numerical results for the Dirac equation \eqref{SDEd} (or
(\ref{SDEdd})) in the literatures. Along the analytical front,
time independent states and dynamical properties
have been  thoroughly investigated, such as the bound states  \cite{Esteban},
semi-classical limit \cite{GMMP} and nonrelativistic limit \cite{Sch,N,BMP}, etc.
Along the numerical front, various finite difference time domain (FDTD) methods
 \cite{Ant,XST, WT,Hamm0,Hamm}, time-splitting Fourier pseudospectral (TSFP) method \cite{BL,HJMSZ}
 and  Gaussian beam method \cite{WHJY} have been proposed to
 solve the Dirac equation \eqref{SDEd} (or \eqref{SDEdd}). However, most existing numerical methods
 are designed for the efficient and accurate simulation of the Dirac equation \eqref{SDEd}
 (or \eqref{SDEdd}) in the parameter regime $\eps=O(1)$.
 In fact, for the Dirac equation  in the nonrelativistic limit regime,
 i.e. $0<\eps\ll1$, based on the theoretical analysis
\cite{Hun, BMP,CC,NM,Mau,N,Sch,White}, the solution
exhibits propagating waves with wavelength $O(\varepsilon^2)$ and
$O(1)$ in time and space, respectively. This rapid oscillation
in time brings significant difficulties in designing and analyzing
numerical methods for the Dirac equation  \eqref{SDEd} (or \eqref{SDEdd}) when $0<\eps\ll1$.
Recently, we have rigorously analyzed and compared the frequently used FDTD methods and TSFP method
for the Dirac equation in the nonrelativistic limit regime \cite{BCJ} and
shown that the meshing strategy (or $\eps$-resolution) for the FDTD methods and TSFP method should be $h=O(\sqrt{\eps})$, $\tau=O(\eps^3)$ and $h=O(1)$, $\tau=O(\eps^2)$, respectively,
where $h$ is the mesh size and $\tau$ is the time step. Thus, the existing FDTD
and TSFP methods are capable to solve the Dirac equation \eqref{SDEd} (or \eqref{SDEdd})
efficiently and accurately in the regime $\eps=O(1)$, and are much less efficient
and time consuming in the nonrelativistic limit regime $0<\eps\ll1$.

The main aim of this paper is to design and analyze
an efficient and accurate numerical method for the Dirac equation \eqref{SDEd} (or (\ref{SDEdd}))
which is uniformly accurate (UA) for $\eps\in(0,1]$.
The key ingredients include  adopting a multiscale decomposition of
the solution of the Dirac equation \cite{Masmoudi} at each time interval
with proper transmission conditions at different time intervals and
 applying the exponential wave integrator (EWI) with appropriate numerical quadratures
which have been widely explored in solving highly oscillatory ordinary differential equations (ODEs) \cite{Gautschi0,Gautschi-type-3,Gautschi-type-5} and dispersive partial differential equations (PDEs) \cite{BC2,BD,BDZ}.
Then by a careful study of the error propagation and using the energy method, we
establish two independent error estimates via two different mathematical
approaches for the MTI-FP method
as $h^{m_0}+\frac{\tau^2}{\eps^2}$ and $h^{m_0}+\tau^2+\eps^2$ with $m_0$ depending
on the regularity of the solution.
Thus the MTI-FP method converges uniformly in space and time with
respect to $0<\eps\le1$. We remark here that a similar MTI-FP method
has been recently designed and analyzed for the nonlinear Klein-Gordon equation
in the nonrelativistic limit regime \cite{BCZ}.

The rest of the paper is organized as follows.
In section \ref{sec:MTI}, we introduce a  multiscale decomposition
for the Dirac equation \eqref{SDEdd} and design the MTI-FP method.
In section \ref{sec:error}, we establish rigorously error  estimates for the MTI-FP method.
Section \ref{sec:num} is devoted to the numerical results of the MTI-FP
method and convergence rates of the solution of the Dirac equation to its limiting models
in the nonrelativistic limit regime.
Finally, some conclusions are drawn in section \ref{sec:con}.
Throughout the paper, we adopt standard notations of Sobolev spaces and their norms, and use
the notation $p\lesssim q$ to represent that there exists a generic constant $C>0$,
which is independent of time step $\tau$, mesh size $h$ and $\eps$, such that $|p|\leq Cq$.

\section{The MTI-FP method}\label{sec:MTI}
For simplicity of notations, we shall only present our
method and analysis for the Dirac equation \eqref{SDEdd} in 1D,
while all the notations and results
can be easily generalized to \eqref{SDEdd} in higher dimensions (2D) and \eqref{SDEd} without any extra work. Denote
\be\label{eq:op:def}
\mathcal{T}=-i\eps\sigma_1\partial_x+\sigma_3,\quad W(t,x)=V(t,x)I_2-
A_1(t,x)\sigma_1,\quad x\in \Bbb R,
\ee
where the domain of the operator $\mathcal{T}$ is $(H^1(\Bbb R))^2$,
and
then the Dirac equation  \eqref{SDEdd} in 1D can be written as
\be
i\p_t\Phi(t,x)=\frac{1}{\eps^2}\mathcal{T}\Phi(t,x)+W(t,x)\Phi(t,x).
\ee
We note that $\mathcal{T}$ is diagonalizable in the phase space (Fourier domain) and can be decomposed as
\be
\mathcal{T}=\sqrt{I-\eps^2\Delta}\;\Pi_+-\sqrt{I-\eps^2\Delta}\;\Pi_-,
\ee
where $\Delta=\p_{xx}$ is the Laplace operator   in 1D and $I$ is the identity operator,   $\Pi_+$ and $\Pi_-$ are projectors defined as
\be
\Pi_+=\frac12\left[I_2+\left(I-\eps^2\Delta\right)^{-1/2}\mathcal{T}\right],\quad \Pi_-=\frac12\left[I_2-\left(I-\eps^2\Delta\right)^{-1/2}\mathcal{T}\right].
\ee
Here $\sqrt{I-\eps^2\Delta}$ is understood in the Fourier space by the symbol $\sqrt{1+\eps^2\xi^2}$ ($\xi\in\Bbb R$) with domain $H^1(\Bbb R)$. When $\Delta$, $I$ and $\sqrt{I-\eps^2\Delta}$ act on vector function $\Phi=(\phi_1,\phi_2)^T$, we mean that $\Delta$, $I$ and $\sqrt{I-\eps^2\Delta}$ act on two components $\phi_1$ and $\phi_2$.
It is easy to verify that $\Pi_++\Pi_-=I_2$ and $\Pi_+\Pi_-=\Pi_-\Pi_+={\bf 0}$, $\Pi_{\pm}^2=\Pi_{\pm}$. In addition, $\mathcal{T}$ and $\Pi_{\pm}$ can be easily
calculated in Fourier domain.

As will be shown in the subsequent discussion,
this decomposition of $\mathcal{T}$ is
the key step for designing the uniformly accurate numerical scheme. The idea is to decouple the solutions into the eigenspaces of  $\mathcal{T}$ in phase space
 by applying the projectors $\Pi_{\pm}$, and  deal with the projections separately.
 It is worth noticing that the above formulation is in the whole space,  while it is
 necessary to truncate the problem onto a bounded domain in computation. Thanks to  the Fourier series,  the operator $\mathcal{T}$ and
 its associated eigenvalue decomposition $\Pi_{\pm}$ can be explicitly computed in the phase space for the bounded domain case (see \eqref{eq:dl}, \eqref{eq:Pipm}), which is crucial for the success of our method introduced below.

\subsection{Multiscale decomposition}
In order to design a uniformly accurate numerical method for the
Dirac equation \eqref{SDEdd} (or \eqref{SDEd}),
from the experience in designing uniformly accurate methods for
the nonlinear Klein-Gordon equation in the nonrelativistic limit regime \cite{BCZ,Chartier,FS2},
recalling that there exist propagating waves with $O(\eps^2)$ wavelength in time,
a multiscale decomposition should possess  $O(\eps^2)$ accuracy,
so that the first order time derivative of the residue is
bounded and a uniformly accurate scheme can be obtained.
Thus, the first order Schr\"odinger decomposition \eqref{eq:MD:1st}
is inappropriate, and the second order Pauli-type decomposition (see \cite{N,NM,BMP}) might work.
However, due to the linearity of the Dirac equation \eqref{SDEdd} (or \eqref{SDEd}),
we have a  direct and better decomposition by applying the
projectors $\Pi_{\pm}$ to the Dirac equation \eqref{SDEdd} \cite{BMP}.

Choose the time step $\tau:=\Delta t>0$ and denote time
steps as $t_n:=n\,\tau$ for $n\ge0$. From $t=t_n$ to $t_{n+1}$,
the solution $\Phi(t,x)=\Phi(t_n+s,x)$
(denote $\Phi^n(x)=\Phi(t_n,x)$) to the Dirac equation \eqref{SDEdd} can be decomposed as
\be\label{eq:multideco}
\Phi(t_n+s,x)=e^{-is/\eps^2}\left(\Psi_+^{1,n}(s,x)+\Psi_-^{1,n}(s,x)\right)
+e^{is/\eps^2}\left(\Psi_+^{2,n}(s,x)+\Psi_-^{2,n}(s,x)\right),\quad 0\leq s\leq \tau,
\ee
where $\left(\Psi_+^{1,n}(s,x),\Psi_-^{1,n}(s,x)\right)$
solves the coupled system for $x\in\Bbb R$  and $0\le s\le \tau$ as
\be\label{eq:decsysfull:1}
\begin{cases}
i\p_s\Psi_+^{1,n}(s,x)=\frac{1}{\eps^2}\left(\sqrt{I-\eps^2\Delta}-I\right)\Psi_+^{1,n}(s,x)
+\Pi_+\left(W\Psi_+^{1,n}(s,x)+W\Psi_-^{1,n}(s,x)\right),\\
i\p_s\Psi_-^{1,n}(s,x)=\frac{1}{\eps^2}\left(-\sqrt{I-\eps^2\Delta}-I\right)\Psi_-^{1,n}(s,x)
+\Pi_-\left(W\Psi_+^{1,n}(s,x)+W\Psi_-^{1,n}(s,x)\right),\\
\Psi_+^{1,n}(0,x)=\Pi_+\Phi^n(x),\quad \Psi_-^{1,n}(0,x)={\bf 0},
\end{cases}
\ee
with $W:=W(t_n+s,x)$, and similarly $\left(\Psi_+^{2,n}(s,x),\Psi_-^{2,n}(s,x)\right)$ satisfies
\be\label{eq:decsysfull:2}
\begin{cases}
i\p_s\Psi_+^{2,n}(s,x)=\frac{1}{\eps^2}\left(\sqrt{I-\eps^2\Delta}+I\right)\Psi_+^{2,n}(s,x)
+\Pi_+\left(W\Psi_+^{2,n}(s,x)+W\Psi_-^{2,n}(s,x)\right),\\
i\p_s\Psi_-^{2,n}(s,x)=\frac{1}{\eps^2}\left(-\sqrt{I-\eps^2\Delta}+I\right)\Psi_-^{2,n}(s,x)
+\Pi_-\left(W\Psi_+^{2,n}(s,x)+W\Psi_-^{2,n}(s,x)\right),\\
\Psi_+^{2,n}(0,x)={\bf 0},\quad \Psi_-^{2,n}(0,\bx)=\Pi_-\Phi^n(x).
\end{cases}
\ee
Following the analysis in \cite{BMP}, it is easy to verify that  $\Psi_+^{1,n}(s,x)=O(1),\Psi_-^{2,n}(s,x)=O(1),\Psi_-^{1,n}(s,x)=O(\eps^2),\Psi_+^{2,n}(s,x)=O(\eps^2)$,
and $\p_s\Psi_{\pm}^{k,n}=O(1)$ for $k=1,2$. Thus $\Phi(t_{n+1},x)$ can be evaluated numerically by
solving  the two coupled systems \eqref{eq:decsysfull:1} and \eqref{eq:decsysfull:2}
via the exponential wave integrator Fourier pseudospectral (EWI-FP) method\cite{BC2,BCZ}
 through the decomposition \eqref{eq:multideco}.

\subsection{The MTI Fourier spectral discretization}
As a common practice in the literatures
\cite{BL, BHM, Hamm0, Hamm, HJMSZ, NSG,  WHJY, WT} for practical computation,
the Dirac equation \eqref{SDEdd} with $d=1$ is usually truncated on
a bounded  interval $\Omega=(a,b)$
for $\Phi:=\Phi(t,x)\in \mathbb{C}^2$,
 \be\label{eq:de1d}
 i\partial_t\Phi(t,x)=\frac{1}{\eps^2}\mathcal{T}\Phi(t,x)+W(t,x)\Phi(t,x),
\qquad x\in\Omega,\quad t>0,
\ee
with periodic boundary conditions and initial condition as
\be\label{eq:de1d-init}
\Phi(t,x)\ \hbox{is}\ (b-a)\ \hbox{periodic in}\ x, \quad t\geq 0;\qquad
\Phi(0,x) = \Phi_0(x),\quad x\in\overline{\Omega};
\ee
where we use the same notations $\mathcal{T}$ and $W(t,x)$  here as  those in the whole space case (\ref{eq:op:def}) by abuse of notation and we remark that the
domain of $\mathcal{T}$ here is $(H_p^1(\Omega))^2$ with $H_p^1(\Omega)=\left\{u\in H^1(\Omega)| u(a)=u(b)\right\}$.

Then the systems \eqref{eq:decsysfull:1} and \eqref{eq:decsysfull:2} for
the decomposition  \eqref{eq:multideco}  with $x\in\Omega$ and $0\le s\le \tau$
collapse to
\be\label{eq:decsys:1}
\begin{cases}
i\p_s\Psi_+^{1,n}(s,x)=\frac{1}{\eps^2}\left(\sqrt{I-\eps^2\Delta}-I\right)\Psi_+^{1,n}(s,x)
+\Pi_+\left(W\Psi_+^{1,n}(s,x)+W\Psi_-^{1,n}(s,x)\right),\\
i\p_s\Psi_-^{1,n}(s,x)=\frac{1}{\eps^2}\left(-\sqrt{I-\eps^2\Delta}-I\right)\Psi_-^{1,n}(s,x)
+\Pi_-\left(W\Psi_+^{1,n}(s,x)+W\Psi_-^{1,n}(s,x)\right),\\
\Psi_{\pm}^{1,n}(s,x) \ \hbox{is}\ (b-a)\ \hbox{periodic in}\ x, \quad
\Psi_+^{1,n}(0,x)=\Pi_+\Phi(t_n,x),\quad \Psi_-^{1,n}(0,x)={\bf0},
\end{cases}
\ee
and
\be\label{eq:decsys:2}
\begin{cases}
i\p_s\Psi_+^{2,n}(s,x)=\frac{1}{\eps^2}\left(\sqrt{I-\eps^2\Delta}+I\right)\Psi_+^{2,n}(s,x)
+\Pi_+\left(W\Psi_+^{2,n}(s,x)+W\Psi_-^{2,n}(s,x)\right),\\
i\p_s\Psi_-^{2,n}(s,x)=\frac{1}{\eps^2}\left(-\sqrt{I-\eps^2\Delta}+I\right)\Psi_-^{2,n}(s,x)
+\Pi_-\left(W\Psi_+^{2,n}(s,x)+W\Psi_-^{2,n}(s,x)\right),\\
\Psi_{\pm}^{2,n}(s,x)  \ \hbox{is}\ (b-a)\ \hbox{periodic in}\ x,\quad
\Psi_+^{2,n}(0,x)={\bf0},\quad \Psi_-^{2,n}(0,x)=\Pi_-\Phi(t_n,x).
\end{cases}
\ee

Choose the mesh size $h:=\Delta x=\frac{b-a}{M}$ with $M$  being a positive even integer and denote
the grid points
as $x_j:=a+j\,h$ for $j=0,1,\ldots, M$.
Denote $X_M=\{U=(U_0,U_1,...,U_M)^T\ |\ U_j\in {\mathbb C}^2, j=0,1,\ldots,M, \ U_0=U_M\}$
and the $l^2$ norm in $X_M$ is given by
\be
\|U\|^2_{l^2}=h\sum^{M-1}_{j=0}|U_j|^2, \qquad  U\in X_M.
\ee
 Introduce
\begin{equation*}
Y_{M}=Z_M\times Z_M, \qquad \hbox{with}\quad
Z_M={\rm span}\left\{\phi_l(x)=e^{i\mu_l(x-a)},\,\,\mu_l=\frac{2l\pi}{b-a},\quad l=-\frac{M}{2}, \ldots, \frac{M}{2}-1\right\}.
\end{equation*}
Let $[C_p(a,b)]^2$ be the function space consisting of all periodic
vector function $U(x):\ [a,b]\to {\mathbb C}^2$. For any $U(x)\in [L^2(a,b)]^2$ and $U\in X_M$,
define $P_M:\ [L^2(a, b)]^2\rightarrow Y_M$ as the standard projection operator,
$I_M:\ [C_p(a, b)]^2 \rightarrow Y_M$ and $I_M:X_M\rightarrow Y_M$ as the standard
interpolation operator, i.e.
\be
(P_MU)(x)=\sum_{l=-M/2}^{M/2-1}\widehat{U}_l\, e^{i\mu_l(x-a)},
\quad (I_MU)(x)=\sum_{l=-M/2}^{M/2-1}\widetilde{U}_l\, e^{i\mu_l(x-a)},\qquad
a\leq x\leq b,
\ee
with
\be\label{fouriercoef}
\widehat{U}_l=\frac{1}{b-a}\int_a^bU(x)\,e^{-i\mu_l(x-a)}\,dx,\quad
\widetilde{U}_l=\frac{1}{M}\sum_{j=0}^{M-1}U_j\, e^{-2ijl\pi/M},\quad l=-\frac{M}{2},\ldots, \frac{M}{2}-1,
\ee
where $U_j=U(x_j)$ when $U$ is a function. The Parseval's identity implies that
\be\label{eq:normeq}
\|I_M(U)(\cdot)\|_{L^2}=\|U\|_{l^2},\quad \forall U\in X_M.
\ee

The Fourier spectral discretization for \eqref{eq:decsys:1}-\eqref{eq:decsys:2} reads:

\noindent Find $\Psi_{\pm,M}^{k,n}:=\Psi_{\pm,M}^{k,n}(s)=\Psi_{\pm,M}^{k,n}(s,x)\in Y_M$ ($0\leq s\leq \tau$), i.e.
\begin{eqnarray}
\label{FPR}
\Psi_{\pm,M}^{k,n}(s,x)=\sum^{M/2-1}_{l=-M/2}\widehat{\left(\Psi_{\pm}^{k,n}\right)}_l(s)\,
e^{i\mu_l(x-a)},\quad a\le x\le b, \qquad s\ge0,\qquad k=1,2,
\end{eqnarray}
such that for $a<x<b$ and $0\le s\le \tau$
\be\label{eq:decsys:FS:1}
\begin{cases}
i\p_{s}\Psi_{+,M}^{1,n}(s)=\frac{1}{\eps^2}\left(\sqrt{I-\eps^2\Delta}-I\right)\Psi_{+,M}^{1,n}(s)
+\Pi_+\left(W\Psi_{+,M}^{1,n}(s)+W\Psi_{-,M}^{1,n}(s)\right),\\
i\p_s\Psi_{-,M}^{1,n}(s)=\frac{1}{\eps^2}\left(-\sqrt{I-\eps^2\Delta}-I\right)\Psi_{-,M}^{1,n}(s)
+\Pi_-\left(W\Psi_{+,M}^{1,n}(s)+W\Psi_{-,M}^{1,n}(s)\right),\\
\Psi_{+,M}^{1,n}(0)=P_M\left(\Pi_+\Phi(t_n,x)\right),\quad \Psi_{-,M}^{1,n}(0)={\bf 0},
\end{cases}
\ee
and \be\label{eq:decsys:FS:2}
\begin{cases}
i\p_{s}\Psi_{+,M}^{2,n}(s)=\frac{1}{\eps^2}\left(\sqrt{I-\eps^2\Delta}+I\right)\Psi_{+,M}^{2,n}(s)
+\Pi_+\left(W\Psi_{+,M}^{2,n}(s)+W\Psi_{-,M}^{2,n}(s)\right),\\
i\p_s\Psi_{-,M}^{2,n}(s)=\frac{1}{\eps^2}\left(-\sqrt{I-\eps^2\Delta}+I\right)\Psi_{-,M}^{2,n}(s)
+\Pi_-\left(W\Psi_{+,M}^{2,n}(s)+W\Psi_{-,M}^{2,n}(s)\right),\\
\Psi_{+,M}^{2,n}(0)={\bf 0},\quad \Psi_{-,M}^{2,n}(0)=P_M\left(\Pi_-\Phi(t_n,x)\right).
\end{cases}
\ee
We then obtain the equations for the Fourier coefficients with $0\le s\le \tau$ as
\be\label{eq:fc:1}\begin{cases}
i\p_s\widehat{\left(\Psi_{+}^{1,n}\right)}_l(s)=\frac{\delta_l^-}{\eps^2}
I_2\widehat{\left(\Psi_{+}^{1,n}\right)}_l(s)
+\Pi_l^+\widehat{\left(W\Psi_{+,M}^{1,n}\right)}_l(s)+\Pi_l^+\widehat{\left(W\Psi_{-,M}^{1,n}\right)}_l(s),\\
i\p_s\widehat{\left(\Psi_-^{1,n}\right)}_l(s)=-\frac{\delta_l^+}
{\eps^2}I_2\widehat{\left(\Psi_-^{1,n}\right)}_l(s)
+\Pi_l^-\widehat{\left(W\Psi_{+,M}^{1,n}\right)}_l(s)+\Pi_l^-\widehat{\left(W\Psi_{-,M}^{1,n}\right)}_l(s),
\end{cases}
\ee
and
\be\label{eq:fc:2}\begin{cases}
i\p_s\widehat{\left(\Psi_+^{2,n}\right)}_l(s)=
\frac{\delta_l^+}{\eps^2}I_2\widehat{\left(\Psi_{+}^{2,n}\right)}_l(s)
+\Pi_l^+\widehat{\left(W\Psi_{+,M}^{2,n}\right)}_l(s)+\Pi_l^+\widehat{\left(W\Psi_{-,M}^{2,n}\right)}_l(s),\\
i\p_s\widehat{\left(\Psi_-^{2,n}\right)}_l(s)=
-\frac{\delta_l^-}{\eps^2}I_2\widehat{\left(\Psi_-^{2,n}\right)}_l(s)
+\Pi_l^-\widehat{\left(W\Psi_{+,M}^{2,n}\right)}_l(s)+\Pi_l^-\widehat{\left(W\Psi_{-,M}^{2,n}\right)}_l(s),
\end{cases}
\ee
for $l=-\frac{M}{2},\ldots,\frac{M}{2}-1$, where
\be\label{eq:dl}
\delta_l=\sqrt{1+\eps^2\mu_l^2},\qquad \delta_l^+=\delta_l+1,\qquad \delta_l^-=\delta_l-1;
\ee
and $\Pi_l^+$ and $\Pi_l^-$ are the corresponding Fourier representations of the projectors $\Pi_{\pm}$ as
\be\label{eq:Pipm}
\Pi_l^{+}=\begin{pmatrix}\frac{1+\delta_l}{2\delta_l}&\frac{\eps\mu_l}{2\delta_l}\\ \frac{\eps\mu_l}{2\delta_l}&\frac{\eps^2\mu_l^2}{2\delta_l(\delta_l+1)}\end{pmatrix},\quad
\Pi_l^{-}=\begin{pmatrix}\frac{\eps^2\mu_l^2}{2\delta_l(\delta_l+1)}&-\frac{\eps\mu_l}{2\delta_l}\\ -\frac{\eps\mu_l}{2\delta_l}&\frac{1+\delta_l}{2\delta_l}\end{pmatrix},\quad  l=-\frac{M}{2},\ldots,\frac{M}{2}-1.
\ee
Using the variation-of-constant formula, we can write the solution as
\begin{equation*}\begin{split}
&\widehat{\left(\Psi_+^{1,n}\right)}_l(s)=e^{-i\delta_l^-s/\eps^2}\widehat{\left(\Psi_+^{1,n}\right)}_l(0)
-i\int_0^se^{-i\delta_l^-(s-w)/\eps^2}\Pi_l^+
\left(\widehat{\left(W\Psi_{+,M}^{1,n}\right)}_l(w)+\widehat{\left(W\Psi_{-,M}^{1,n}\right)}_l(w)\right)\,dw,\\
&\widehat{\left(\Psi_-^{1,n}\right)}_l(s)=e^{i\delta_l^+s/\eps^2}\widehat{\left(\Psi_-^{1,n}\right)}_l(0)
-i\int_0^se^{i\delta_l^+(s-w)/\eps^2}\Pi_l^-
\left(\widehat{\left(W\Psi_{+,M}^{1,n}\right)}_l(w)+\widehat{\left(W\Psi_{-,M}^{1,n}\right)}_l(w)\right)\,dw,\\
&\widehat{\left(\Psi_+^{2,n}\right)}_l(s)=e^{-i\delta_l^+s/\eps^2}\widehat{\left(\Psi_+^{2,n}\right)}_l(0)
-i\int_0^se^{-i\delta_l^+(s-w)/\eps^2}\Pi_l^+
\left(\widehat{\left(W\Psi_{+,M}^{2,n}\right)}_l(w)+\widehat{\left(W\Psi_{-,M}^{2,n}\right)}_l(w)\right)\,dw,\\
&\widehat{\left(\Psi_-^{2,n}\right)}_l(s)=e^{i\delta_l^-s/\eps^2}\widehat{\left(\Psi_-^{2,n}\right)}_l(0)
-i\int_0^se^{i\delta_l^-(s-w)/\eps^2}\Pi_l^-
\left(\widehat{\left(W\Psi_{+,M}^{2,n}\right)}_l(w)+\widehat{\left(W\Psi_{-,M}^{2,n}\right)}_l(w)\right)\,dw.
\end{split}
\end{equation*}
Using the initial conditions and choosing $s=\tau$, we can approximate the integrals via the
Gautschi type quadrature rules \cite{Gautschi0,Gautschi-type-3,Gautschi-type-5,BCZ} or
EWI \cite{BC2,BCJ,BCZ,BD,BDZ}.
Using Taylor expansion and  the equations \eqref{eq:decsys:FS:1}-\eqref{eq:decsys:FS:2}
to determine the first order derivative, we can approximate
the first integral in the above equation as
\begin{align}
\label{eq:inte1}&-i\int_0^\tau e^{-i\delta_l^-(\tau-w)/\eps^2}\Pi_l^{+}\left(\widehat{
\left(W\Psi_{+,M}^{1,n}\right)}_l(w)+\widehat{\left(W\Psi_{-,M}^{1,n}\right)}_l(w)\right)\,dw\\
&\ \approx-i\int_0^\tau e^{-i\delta_l^-(\tau-w)/\eps^2}\Pi_l^{+}
\left(\widehat{\left(W\Psi_{+,M}^{1,n}\right)}_l(0)+\widehat{
\left(W\Psi_{-,M}^{1,n}\right)}_l(0)\right)\,dw\nonumber\\
&\quad -i\int_0^\tau e^{-i\delta_l^-(\tau-w)/\eps^2}w\Pi_l^{+}
\left(\partial_s\widehat{\left(W\Psi_{+,M}^{1,n}\right)}_l(0)+
\partial_s\widehat{\left(W\Psi_{-,M}^{1,n}\right)}_l(0)\right)\,dw\nonumber\\
&\ =p_l^-(\tau)\;\Pi_l^+\widehat{\left(W\Psi_{+,M}^{1,n}\right)}_l(0)+q_l^-
(\tau)\;\Pi_l^+\left(\partial_s\widehat{\left(W\Psi_{+,M}^{1,n}\right)}_l(0)+\partial_s\widehat{\left(
W\Psi_{-,M}^{1,n}\right)}_l(0)\right),\nonumber
\end{align}
where
\be\label{eq:coeff1}
p_l^-(\tau)=-i\tau e^{\frac{-i\tau\delta_l^-}{2\eps^2}}\sinc\left(\frac{\delta_l^-\tau}{2\eps^2}\right),\quad
q_l^-(\tau)=-\frac{\tau\eps^2}{\delta_l^-}\left(I-e^{\frac{-i\tau\delta_l^-}
{2\eps^2}}\sinc\left(\frac{\delta_l^-\tau}{2\eps^2}\right)\right),
\ee
and $\sinc(s)=\frac{\sin s}{s}$ with $\sinc(0)=1$. Note that $p_l^-(\tau)=O(\tau)$ and $q_l^-(\tau)=O(\tau^2)$
for $l=-\frac{M}{2},\ldots,\frac{M}{2}-1$, and for the special case $l=0$, $p_0^-(\tau)=-i\tau$ and $q_0^-(\tau)=-i\frac{\tau^2}{2}$. Similarly, the other integrals can be approximated as
\begin{align}\label{eq:inte2}
&-i\int_0^{\tau}e^{i\delta_l^+(\tau-w)/\eps^2}\Pi_l^-\left(\widehat{
\left(W\Psi_{+,M}^{1,n}\right)}_l(w)+\widehat{\left(W\Psi_{-,M}^{1,n}\right)}_l(w)\right)\,dw\\
&\ \approx-\overline{p_l^+(\tau)}\;\Pi_l^-\widehat{\left(W\Psi_{+,M}^{1,n}\right)}_l(0)-\overline{q_l^+(\tau)}\;
\Pi_l^-\left(\partial_s\widehat{\left(W\Psi_{+,M}^{1,n}\right)}_l(0)+
\partial_s\widehat{\left(W\Psi_{-,M}^{1,n}\right)}_l(0)\right),
\nonumber\\
\label{eq:inte3}&-i\int_0^{\tau}e^{-i\delta_l^+(\tau-w)/\eps^2}\Pi_l^+
\left(\widehat{\left(W\Psi_{+,M}^{2,n}\right)}_l(w)+\widehat{\left(W\Psi_{-,M}^{2,n}\right)}_l(w)\right)\,dw\\
&\ \approx p_l^+(\tau)\;\Pi_l^+\widehat{\left(W\Psi_{-,M}^{2,n}\right)}_l(0)+
q_l^+(\tau)\;\Pi_l^+\left(\partial_s\widehat{\left(W\Psi_{+,M}^{2,n}\right)}_l(0)
+\partial_s\widehat{\left(W\Psi_{-,M}^{2,n}\right)}_l(0)\right),\nonumber\\
\label{eq:inte4}&-i\int_0^{\tau}e^{i\delta_l^-(\tau-w)/\eps^2}\Pi_l^-
\left(\widehat{\left(W\Psi_{+,M}^{2,n}\right)}_l(w)+\widehat{\left(W\Psi_{-,M}^{2,n}\right)}_l(w)\right)\,dw\\
&\ \approx-\overline{p_l^-(\tau)}\;\Pi_l^-\widehat{\left(W\Psi_{-,M}^{2,n}\right)}_l(0)-\overline{q_l^-(\tau)}
\;\Pi_l^-\left(\partial_s\widehat{\left(W\Psi_{+,M}^{2,n}\right)}_l(0)+
\partial_s\widehat{\left(W\Psi_{-,M}^{2,n}\right)}_l(0)\right),\nonumber
\end{align}
with $\bar{c}$ denoting the complex conjugate of $c$ and
\be\label{eq:coeff2}
p_l^+(\tau)=-i\tau e^{\frac{-i\tau\delta_l^+}{2\eps^2}}\sinc\left(\frac{\delta_l^+\tau}{2\eps^2}\right),\quad
q_l^+(\tau)=-\frac{\tau\eps^2}{\delta_l^+}\left(I-e^{\frac{-i\tau\delta_l^+}
{2\eps^2}}\sinc\left(\frac{\delta_l^+\tau}{2\eps^2}\right)\right).
\ee
For simplicity of notations, by omitting the spatial $x$ variable and denoting
\be\label{eq:frldef}
f_{\pm}^{k,n}(s)=W(t_n+s)\Psi_{\pm,M}^{k,n}(s),\quad \dot{f}_{\pm}^{k,n}(s)
=W(t_n)\p_s\Psi_{\pm,M}^{k,n}(s),\quad g_{\pm}^{k,n}(s)=\p_sW(t_n+s)\Psi_{\pm,M}^{k,n}(s),
\ee
in order to design a uniform convergent method with respect to $\eps\in(0,1]$,
we find the solutions should be updated in the order from small component to large component as
\begin{equation*}
\begin{cases}
\widehat{\left(\Psi_-^{1,n}\right)}_l(\tau)&\approx -\overline{p_l^+(\tau)}\Pi_l^-\widehat{\left(f_+^{1,n}\right)}_l(0)-
\overline{q_l^+(\tau)}\Pi_l^-\widehat{\left(g_+^{1,n}\right)}_l(0)-
\overline{q_l^+(\tau)}\Pi_l^-\left(\widehat{\left(\dot{f}_+^{1,n}\right)}_l(0)
+\widehat{\left(\dot{f}_-^{1,n}\right)}_l(0)\right),\\
\widehat{\left(\Psi_+^{2,n}\right)}_l(\tau)&\approx p_l^+(\tau)\Pi_l^+\widehat{\left(f_-^{2,n}\right)}_l(0)+q_l^+(\tau)\Pi_l^
+\widehat{\left(g_-^{2,n}\right)}_l(0)+
q_l^+(\tau)\Pi_l^+\left(\widehat{\left(\dot{f}_+^{2,n}\right)}_l(0)+
\widehat{\left(\dot{f}_-^{2,n}\right)}_l(0)\right),
\end{cases}
\end{equation*}
with initial values and derivatives determined from \eqref{eq:decsys:FS:1}-\eqref{eq:decsys:FS:2} as
\begin{align*}
\widehat{\left(\Psi_+^{1,n}\right)}_l(0)&=\Pi_l^+\widehat{(\Phi(t_n))}_l,\quad \widehat{\left(\Psi_-^{1,n}\right)}_l(0)=0,\quad\widehat{\left(\Psi_+^{2,n}\right)}_l(0)=0, \quad \widehat{\left(\Psi_-^{1,n}\right)}_l(0)=\Pi_l^-\widehat{(\Phi(t_n))}_l,\\
\widehat{\left(\p_s\Psi_{+,M}^{1,n}\right)}_l(0)&\approx-i\frac{2\sin(\mu_l^2\tau/2)}
{\delta_l^+\tau}\widehat{\left(\Psi_+^{1,n}\right)}_l(0)
-i\;\Pi_l^+\widehat{\left(W(t_n)\Psi_{+,M}^{1,n}(0)\right)}_l,\\
\widehat{\left(\p_s\Psi_{-,M}^{1,n}\right)}_l(0)&= -i\;\Pi_l^-
\widehat{\left(W(t_n)\Psi_{+,M}^{1,n}(0)\right)}_l,\quad
\widehat{\left(\p_s\Psi_{+,M}^{2,n}\right)}_l(0)= -i\;\Pi_l^+
\widehat{\left(W(t_n)\Psi_{-,M}^{2,n}(0)\right)}_l,\\
\widehat{\left(\p_s\Psi_{-,M}^{2,n}\right)}_l(0)&\approx i\frac{2\sin(\mu_l^2\tau/2)}{\delta_l^+\tau}\widehat{\left(\Psi_{-}^{2,n}\right)}_l(0)
-i\;\Pi_l^-\widehat{\left(W(t_n)\Psi_{-,M}^{2,n}(0)\right)}_l,\quad  l=-\frac{M}{2},\ldots,\frac{M}{2}-1,
\end{align*}
where the derivatives $\p_s\Psi_{+,M}^{1,n}(0)$ and $\p_s\Psi_{-,M}^{2,n}(0)$ are approximated using filters $2\sin(\mu_l^2\tau/2)/\tau$ instead of $\mu_l^2$ ($l=-\frac{M}{2},\ldots,\frac{M}{2}-1$)
to avoid loss of accuracy, then followed by
\begin{equation*}
\begin{cases}
\widehat{\left(\Psi_+^{1,n}\right)}_l(\tau)&\approx e^{-i\delta_l^-\tau/\eps^2}\widehat{\left(\Psi_+^{1,n}\right)}_l(0)+p_l^-
(\tau)\;\Pi_l^+\widehat{\left(f_+^{1,n}\right)}_l(0)+
q_l^-(\tau)\;\Pi_l^+\widehat{\left(g_+^{1,n}\right)}_l(0)\\
&\quad +q_l^-(\tau)\;\Pi_l^+\left(\widehat{\left(\dot{f}_+^{1,n}\right)}_l(0)+
\widehat{\left(\dot{f}_-^{1,n}\right)}_l(0)\right),\\
\widehat{\left(\Psi_-^{2,n}\right)}_l(\tau)&\approx e^{i\delta_l^-
\tau/\eps^2}\widehat{\left(\Psi_-^{2,n}\right)}_l(0)-
\overline{p_l^-(\tau)}\;\Pi_l^-\widehat{\left(f_-^{2,n}\right)}_l(0)
-\overline{q_l^-(\tau)}\;\Pi_l^-\widehat{\left(g_-^{2,n}\right)}_l(0)\\
&\quad-
\overline{q_l^-(\tau)}\;\Pi_l^-\left(\widehat{\left(\dot{f}_+^{2,n}\right)}_l(0)+
\widehat{\left(\dot{f}_-^{2,n}\right)}_l(0)\right),
\end{cases}
\end{equation*}
with $\widehat{\left(\p_s\Psi_-^{1,n}\right)}_l(0)$ and
$\widehat{\left(\p_s\Psi_+^{2,n}\right)}_l(0)$ approximated in another way as
\begin{equation*}
\widehat{\left(\p_s\Psi_-^{1,n}\right)}_l(0)\approx
\frac{1}{\tau}\widehat{\left(\Psi_-^{1,n}\right)}_l(\tau),
\quad \widehat{\left(\p_s\Psi_+^{2,n}\right)}_l(0)\approx
\frac{1}{\tau}\widehat{\left(\Psi_+^{2,n}\right)}_l(\tau),\quad  l=-\frac{M}{2},\ldots,\frac{M}{2}-1.
\end{equation*}

\subsection{The MTI-FP method in 1D}
In practice, the integrals for computing the Fourier transform
coefficients in (\ref{fouriercoef}), (\ref{eq:inte1})-(\ref{eq:inte4})
are usually approximated by the numerical quadratures.
Let $\Phi_j^n$ be the numerical approximation of the
exact solution $\Phi(t_n,x_j)$ to the Dirac equation \eqref{eq:de1d}
for $n\ge0$ and $j=0,1,\ldots,M$, and denote $\Phi^n\in X_M$ as the numerical solution vector
at time $t=t_n$; in addition, let $\Psi_{\pm,j}^{k,n+1}$ ($k=1,2$) be the numerical approximation of
$\Psi_{\pm}^{k,n}(\tau,x_j)$ for $j=0,1,\ldots,M$ and $n\ge0$, and denote
$V_j^n=V(t_n,x_j)$, $A_{1,j}^n=A_1(t_n,x_j)$ and
$W_j^n=W(t_n,x_j)=V_j^nI_2-A_{1,j}^n\sigma_1$  for $j=0,1,\ldots,M$ and $n\ge0$.
Choosing $\Phi^0_j=\Phi_0(x_j)$ for $j=0,1,\ldots,M$,
then the multiscale time integrator Fourier pseudospectral (MTI-FP) method
for discretizing the Dirac equation (\ref{SDEdd}) in 1D reads for $n\ge0$ and $j=0,1,\ldots,M$ as:
\be\label{eq:scheme:1}
\Phi_j^{n+1}=e^{-i\tau/\eps^2}\left(\Psi_{+,j}^{1,n+1}+\Psi_{-,j}^{1,n+1}\right)
+e^{i\tau/\eps^2}\left(\Psi_{+,j}^{2,n+1}+\Psi_{-,j}^{2,n+1}\right)=\sum\limits_{l=-M/2}^{M/2-1}
\widetilde{\left(\Phi^{n+1}\right)}_le^{i\mu_l(x_j-a)},
\ee
where
\be\label{eq:scheme:2}
\Psi_{\pm,j}^{k,n+1}=\sum\limits_{l=-M/2}^{M/2-1}\widetilde{\left(\Psi_{\pm}^{k,n+1}
\right)}_le^{i\mu_l(x_j-a)},\quad k=1,2,
\ee
with
\begin{equation}\label{eq:scheme:3}
\begin{cases}
\widetilde{\left(\Psi_-^{1,n+1}\right)}_l&= -\overline{p_l^+(\tau)}\Pi_l^-
\widetilde{\left(f_+^{1}\right)}_l-
\overline{q_l^+(\tau)}\Pi_l^-\widetilde{\left(g_+^{1}\right)}_l-
\overline{q_l^+(\tau)}\Pi_l^-\left(\widetilde{\left(\dot{f}_+^{1}\right)}_l+
\widetilde{\left(\dot{f}_-^{1}\right)}_l\right),\\
\widetilde{\left(\Psi_+^{2,n+1}\right)}_l&= p_l^+(\tau)\Pi_l^+
\widetilde{\left(f_-^{2}\right)}_l+q_l^+(\tau)\Pi_l^+
\widetilde{\left(g_-^{2}\right)}_l+
q_l^+(\tau)\Pi_l^+\left(\widetilde{\left(\dot{f}_+^{2}\right)}_l
+\widetilde{\left(\dot{f}_-^{2}\right)}_l\right),\\
\widetilde{\left(\Psi_+^{1,n+1}\right)}_l&= e^{-i\frac{\delta_l^-
\tau}{\eps^2}}\widetilde{\left(\Psi_+^{1}\right)}_l+
p_l^-(\tau)\Pi_l^+\widetilde{\left(f_+^{1}\right)}_l+
q_l^-(\tau)\Pi_l^+\widetilde{\left(g_+^{1}\right)}_l +q_l^-(\tau)\Pi_l^+
\left(\widetilde{\left(\dot{f}_+^{1}\right)}_l+
\widetilde{\left(\dot{f}_-^{1,*}\right)}_l\right),\\
\widetilde{\left(\Psi_-^{2,n+1}\right)}_l&= e^{i\frac{\delta_l^-
\tau}{\eps^2}}\widetilde{\left(\Psi_-^{2}\right)}_l-
\overline{p_l^-(\tau)}\Pi_l^-\widetilde{\left(f_-^{2}\right)}_l
-\overline{q_l^-(\tau)}\Pi_l^-\widetilde{\left(g_-^{2}\right)}_l-
\overline{q_l^-(\tau)}\Pi_l^-\left(\widetilde{\left(\dot{f}_+^{2,*}\right)}_l+
\widetilde{\left(\dot{f}_-^{2}\right)}_l\right),
\end{cases}
\end{equation}
and
\begin{equation}\label{eq:scheme:4}
\begin{cases}
f_{\pm,j}^{k}=\sum\limits_{l=-M/2}^{M/2-1}\widetilde{\left(f_{\pm}^{k}
\right)}_le^{i\mu_l(x_j-a)},\quad
\dot{f}_{\pm,j}^{k}=\sum\limits_{l=-M/2}^{M/2-1}\widetilde{\left(
\dot{f}_{\pm}^{k}\right)}_le^{i\mu_l(x_j-a)},\quad
g_{\pm,j}^{k}=\sum\limits_{l=-M/2}^{M/2-1}\widetilde{\left(
g_{\pm}^{k}\right)}_le^{i\mu_l(x_j-a)},\\ \dot{f}_{-,j}^{1,*}=
\sum\limits_{l=-M/2}^{M/2-1}\widetilde{
\left(\dot{f}_{-}^{1,*}\right)}_le^{i\mu_l(x_j-a)}
\quad \dot{f}_{+,j}^{2,*}=\sum\limits_{l=-M/2}^{M/2-1}\widetilde{
\left(\dot{f}_{+}^{2,*}\right)}_le^{i\mu_l(x_j-a)},\quad j=0,1,\ldots,M, \quad k=1,2,
\end{cases}
\end{equation}
with
\begin{equation}\label{eq:scheme:5}
\begin{cases}
\widetilde{\left(\Psi_+^{1}\right)}_l=\Pi_l^+\widetilde{(\Phi^n)}_l,
\qquad \widetilde{\left(\Psi_-^{1}\right)}_l={\bf 0},
\qquad\widetilde{\left(\Psi_+^{2}\right)}_l={\bf 0}, \qquad \widetilde{\left(\Psi_-^{1}\right)}_l=\Pi_l^-\widetilde{(\Phi^n)}_l,\\
\widetilde{\left(\dot\Psi_+^{1}\right)}_l=-i\frac{2\sin(\mu_l^2\tau/2)}
{\delta_l^+\tau}\widetilde{\left(\Psi_+^{1}\right)}_l
-i\,\Pi_l^+\widetilde{\left(f_+^1\right)}_l,\quad
\widetilde{\left(\dot{\Psi}_-^{1}\right)}_l= -i\,\Pi_l^-\widetilde{\left(f_+^{1}\right)}_l,\quad l=-\frac{M}{2},\ldots,\frac{M}{2}-1,\\
\widetilde{\left(\dot{\Psi}_+^{2}\right)}_l= -i\,\Pi_l^+\widetilde{\left(f_-^{2}\right)}_l,\qquad
\widetilde{\left(\dot{\Psi}_-^{2}\right)}_l=i\frac{2\sin(\mu_l^2\tau/2)}
{\delta_l^+\tau}\widetilde{\left(\Psi_-^{2}\right)}_l-i\,\Pi_l^-\widetilde{\left(f_-^{2}\right)}_l,\\
f_{\pm,j}^k=W_j^n\Psi_{\pm,j}^k,\qquad \dot{f}_{\pm,j}^k=W_j^n
\dot{\Psi}_{\pm,j}^k,\qquad g_{\pm,j}^k=\p_tW(t_n,x_j) \Psi_{\pm,j}^k,\\
\dot{f}_{-,j}^{1,*}=\frac{1}{\tau}W_j^n\left(\Psi_{-,j}^{1,n+1}\right)\qquad \dot{f}_{+,j}^{2,*}=\frac{1}{\tau}W_j^n\left(\Psi_{+,j}^{2,n+1}\right), \qquad j=0,1,\ldots, M, \quad k=1,2.
\end{cases}
\end{equation}
\begin{rmk}When the electro-magnetic potentials in the Dirac equation \eqref{eq:de1d} are time independent, i.e. $V(t,x)=V(x)$ and $A_1(t,x)=A_1(x)$,
the above scheme \eqref{eq:scheme:1}-\eqref{eq:scheme:5} will be simplified with  $g_{\pm,j}^k=0$, $k=1,2, j=0,1,\ldots,M$, i.e., the $\widetilde{(g_{\pm}^k)}_l$ term vanishes.
\end{rmk}

Note that in the above MTI-FP method, we first
evaluate the small components $\Psi_{-,j}^{1,n+1}$ and $\Psi_{+,j}^{2,n+1}$,
then approximate their time derivatives via finite difference approximations,
and finally use them in the evaluation of the large components $\Psi_{+,j}^{1,n+1}$
and $\Psi_{-,j}^{2,n+1}$.

  This MTI-FP method for the Dirac equation (\ref{SDEdd}) (or (\ref{SDEd})) is explicit, accurate,
easy to implement and very efficient due to the  discrete fast Fourier transform. The memory
cost is $O(M)$ and the computational cost per time step is $O(M \log M )$.
As will be shown in the next section, it is uniformly convergent
in space and time with respect to $\eps\in(0,1]$.

\section{A uniform error bound}\label{sec:error}
In this section, we establish two independent error estimates for the MTI-FP method
\eqref{eq:scheme:1} via two different mathematical approaches. Let $0<T<\infty$ be any fixed time.
Motivated by the results for the Dirac equation (\ref{SDEdd}) (or (\ref{SDEd}))
in \cite{BMP,CC}, we make the following assumptions
on the electromagnetic potentials
 \begin{equation*}
(A)\hskip2cm \|V\|_{W^{2,\infty}([0,T];(W_p^{m_0,\infty})^2)}+
\|A_1\|_{W^{2,\infty}([0,T];(W_p^{m_0,\infty})^2)}\lesssim 1,\qquad m_0\ge4,\hskip4cm
\end{equation*}
and the exact solution $\Phi:=\Phi(t,x)$ of the Dirac equation \eqref{eq:de1d} with $\eps\in(0,1]$
\begin{align*}
(B)\quad \|\Phi\|_{ L^\infty([0,T]; (H_p^{m_0})^2)}\lesssim 1,\qquad
\|\partial_t\Phi\|_{L^{\infty}([0, T];(H_p^{m_0-1})^2)}\lesssim \frac{1}{\varepsilon^2},\quad
\|\partial_{tt}\Phi(t,x)\|_{L^{\infty}([0,T];(L^2)^2)}\lesssim\frac{1}{\varepsilon^4},\hskip4cm
\end{align*}
where $H^m_p(\Omega)=\{u\ |\ u\in H^{m}(\Omega),\  \partial_x^l u(a)=\partial_x^l u(b),\
l=0,\ldots,m-1\}$ and $W_p^{m,\infty}(\Omega)=\{u\ |\ u\in W^{m,\infty}(\Omega),\
\partial_x^l u(a)=\partial_x^l u(b),\
l=0,\ldots,m-1\}$ for $m\in{\mathbb N}$.
We remark here that the assumption (B) is equivalent to
the initial value $\Phi_0(x)\in (H_p^{m_0})^2$ \cite{N,BMP}
under the assumption (A).

\smallskip

\begin{theorem} \label{thm:main}
Let $\Phi^n\in X_M$ be the
numerical approximation obtained from the MTI-FP method \eqref{eq:scheme:1}
and denote $\Phi_I^n(x)=I_M(\Phi^n)(x)\in Y_M$.
Under the assumptions (A) and (B), there
exist constants $0<\tau_0,h_0\leq1$ sufficiently small and independent of $\eps$,
such that for any $0<\eps\leq1$, when $0<\tau\leq\tau_0$ and $0<h\leq  h_0$, we have
\be\label{eq:est:main}
\|\Phi(t_n,\cdot)-\Phi_I^n(\cdot)\|_{L^2}\lesssim h^{m_0}+\frac{\tau^2}{\eps^2},\quad \|\Phi(t_n,\cdot)-\Phi_I^n(\cdot)\|_{L^2}\lesssim h^{m_0}+\tau^2+\eps^2, \quad 0\le n\le\frac{T}{\tau},
\ee
which yields the uniform error bound by taking
minimum among the two error bounds for $\eps\in(0,1]$
\be
\|\Phi(t_n,\cdot)-\Phi_I^n(\cdot)\|_{L^2}\lesssim h^{m_0}
+\min_{0<\eps\le1} \left\{\frac{\tau^2}{\eps^2},  \tau^2+\eps^2\right\}
\lesssim h^{m_0}+\tau, \quad 0\le n\le\frac{T}{\tau}.
\ee
\end{theorem}
\begin{rmk}
From the analysis point of view, we remark that the $W_p^{m_0,\infty}$
assumption in (A) is necessary such that the exact solution $\Phi(t,x)$ of the Dirac equation
remains in $(H_p^{m_0})^2$, which would give the
spectral accuracy in space. In practice, as long as the solution of the Dirac equation (\ref{SDEdd})
(or (\ref{SDEd}) is well localized such that the error from the
periodic truncation of potential term $W(t,x)\Phi(t,x)$ is negligible,
the error estimates in the above theorem still hold.
\end{rmk}

\bigskip

 In the rest of this paper, we will write the
exact solution $\Phi(t,x)$ as $\Phi(t)$ for simplicity of notations. Define the error function
$\bee^n(x)=\sum\limits_{l=-M/2}^{M/2-1}\widetilde{(\bee^n)}_le^{i\mu_l(x-a)}\in Y_M$ for
$n\ge0$ as
\be\label{eq:edef}
\bee^n(x)=P_M(\Phi(t_n))(x)-\Phi_I^n(x)=P_M(\Phi(t_n))(x)-I_M\left(\Phi^n\right)(x),\quad x\in\Omega.
\ee
Using the assumption (B), triangle inequality and standard Fourier projection properties, we find
\begin{align}\label{eq:error-tri}
\|\Phi(t_n,\cdot)-\Phi_I^n(\cdot)\|_{L^2}\leq&\,\|\Phi(t_n,\cdot)-
P_M(\Phi(t_n))(\cdot)\|_{L^2}+\|P_M(\Phi(t_n))(\cdot)-\Phi_I^n(\cdot)\|_{L^2}\\
\lesssim &\,h^{m_0}+\|\bee^n(\cdot)\|_{L^2},\quad 0\leq n\leq\frac{T}{\tau}.\nonumber
\end{align}
Hence, we only need estimate $\|\bee^n(\cdot)\|_{L^2}$. To this purpose,
local truncation error will be studied as the first step. Since the MTI-FP method \eqref{eq:scheme:1}
is designed by the multiscale decomposition, the following properties of the decomposition \eqref{eq:decsys:1}-\eqref{eq:decsys:2} are essential for the error analysis.

From $t=t_n$ to $t_{n+1}$, let $\Psi_{\pm}^{k,n}(s,x)$ ($0\le s\le \tau$, $k=1,2$) be the solutions of the systems \eqref{eq:decsys:1} and \eqref{eq:decsys:2},
and the decomposition \eqref{eq:multideco} holds as
\be\label{eq:multideco-1d}
\Phi(t_n+s,x)=e^{-is/\eps^2}\left(\Psi_+^{1,n}(s,x)+\Psi_-^{1,n}(s,x)\right)
+e^{is/\eps^2}\left(\Psi_+^{2,n}(s,x)+\Psi_-^{2,n}(s,x)\right),\quad x\in\Omega.
\ee
Then the error $\bee^{n+1}(x)$ ($n\ge0$) \eqref{eq:edef} can be decomposed as
\be\label{eq:dec:error}
\bee^{n+1}(x)=e^{-i\tau/\eps^2}\left(\bez_+^{1,n+1}(x)+\bez_-^{1,n+1}(x)\right)+
e^{i\tau/\eps^2}\left(\bez_+^{2,n+1}(x)+\bez_-^{2,n+1}(x)\right),\quad x\in\Omega,
\ee
with
\be\label{eq:ez-def}
\bez_{\pm}^{k,n+1}(x)=\sum\limits_{l=-M/2}^{M/2-1}\widetilde{(\bez_{\pm}^{k,n+1})}_le^{i\mu_l(x-a)}
=P_M(\Psi_{\pm}^{k,n}(\tau))(x)-I_M\left(\Psi_{\pm}^{k,n+1}\right)(x),\quad k=1,2.
\ee

By the same arguments in \cite{BMP}, we can establish the regularity results.
\begin{lemma}\label{lem:msregularity} Under the assumptions (A) and (B),
the exact solutions $\Psi_{\pm}^{k,n}(s,x)$ ($0\le s\le \tau$,
$0\leq n\leq\frac{T}{\tau}-1$) with $k=1,2$ of the systems \eqref{eq:decsys:1} and
\eqref{eq:decsys:2} satisfy
\begin{align}
\label{eq:reg:1}&\|\Psi_{\pm}^{k,n}\|_{L^\infty([0,\tau];(H_p^{m_0})^2)}\lesssim 1,
\quad \|\partial_{ss}\Psi_{+}^{1,n}\|_{L^\infty([0,\tau];(H_p^{m_4})^2)}
+\|\partial_{ss}\Psi_{-}^{2,n}\|_{L^\infty([0,\tau];(H_p^{m_4})^2)}\lesssim1,\\
\label{eq:reg:2}&\|\partial_s\Psi_{\pm}^{k,n}\|_{L^\infty([0,\tau];(H_p^{m_2})^2)}\lesssim 1,\quad\|\partial_{ss}\Psi_{+}^{2,n}\|_{L^\infty([0,\tau];(H_p^{m_4})^2)}
+\|\partial_{ss}\Psi_{-}^{1,n}\|_{L^\infty([0,\tau];(H_p^{m_4})^2)}\lesssim\frac{1}{\eps^2},\\
\label{eq:reg:3}&\|\Psi_{+}^{2,n}\|_{L^\infty([0,\tau];(H_p^{m_1})^2)}
+\|\Psi_{-}^{1,n}\|_{L^\infty([0,\tau];(H_p^{m_1})^2)}\lesssim\eps^2,\qquad m_l=m_0-l,\quad l=1,2,4.
\end{align}
\end{lemma}

\begin{proof} Noticing the properties of the projectors $\Pi_{\pm}$ and the assumption (B),
we get that the initial data $\Psi_{\pm}^{n}(0,x)\in (H_p^{m_0})^2$,
$\p_s\Psi_{\pm}^{n}(0,x)\in H^{m_0-2}_p$ with uniform bounds.
The  estimates for $\Psi_{\pm}^{k,n}$ and $\p_s\Psi_{\pm}^{k,n}$ have been derived in \cite{BMP},
where one only needs to replace the whole space Fourier transform
with the Fourier series on torus. Thus, the proof is omitted here for brevity. It remains to estimate
$\p_{ss}\Psi_{\pm}^{k,n}$. Here, we show the case $k=1$, while the  case $k=2$
is quite similar and the detail is omitted here for brevity.
Differentiating \eqref{eq:decsys:1} with respect to  $s$, we obtain
for $\p_{ss}\Psi_{\pm}^{1,n}(s)$,
\begin{align*}
i\p_{ss}\Psi_{+}^{1,n}(s)=&-\left(I-\eps^2\Delta\right)^{-1/2}\Delta\p_s\Psi_{+}^{1,n}(s)
+\Pi_+\p_s\left(W\Psi_+^{1,n}(s,x)+W\Psi_-^{1,n}(s,x)\right),\\
i\p_{ss}\Psi_{-}^{1,n}(s)=&-\frac{\sqrt{I-\eps^2\Delta}+I}{\eps^2}\p_s\Psi_{-}^{1,n}(s)
+\Pi_+\p_s\left(W\Psi_+^{1,n}(s,x)+W\Psi_-^{1,n}(s,x)\right).
\end{align*}
Since for any $\Phi\in (H^m_p)^2$ with $2\le m\in {\mathbb N}$, we have
\be
\left\|\left(I-\eps^2\Delta\right)^{-1/2}\Delta\Phi\right\|_{H_p^{m-2}}\leq \|\Phi\|_{H^m_p},\quad
\left\|\frac{\sqrt{I-\eps^2\Delta}+I}{\eps^2}\Phi \right\|_{H_p^{m-1}}\lesssim\frac{1}{\eps^2}\|\Phi\|_{H^m_p},
\ee
which implies the bounds \eqref{eq:reg:2} for $\p_{ss}\Psi_{\pm}^{1,n}(s)$, in view of  the estimates for $\Psi_{\pm}^{1,n}$ and
assumption (A).
\end{proof}

Having Lemma \ref{lem:msregularity} and the decomposition \eqref{eq:dec:error}, we can
define the local truncation error $\xi_{\pm}^{k,n}(x)=\sum\limits_{l=-M/2}^{M/2-1}
\widehat{(\xi_{\pm}^{k,n})}_le^{i\mu_l(x-a)}$ ($x\in\Omega$, $k=1,2$, $n\ge0$)
for the MTI-FP method \eqref{eq:scheme:1}-\eqref{eq:scheme:5} as
\begin{equation}\label{eq:localerr-def}
\begin{cases}
\widehat{(\xi_-^{1,n})}_l=&\widehat{\left(\Psi_-^{1,n}(\tau)\right)}_l +\overline{p_l^+(\tau)}\Pi_l^-\widetilde{\left(f_+^{1}(0)\right)}_l+
\overline{q_l^+(\tau)}\Pi_l^-\widetilde{\left(g_+^{1}(0)\right)}_l+
\overline{q_l^+(\tau)}\Pi_l^-\left(\widetilde{\left(\dot{f}_+^{1}(0)\right)}_l+
\widetilde{\left(\dot{f}_-^{1}(0)\right)}_l\right),\\
\widehat{(\xi_+^{2,n})}_l=&\widehat{\left(\Psi_+^{2,n}(\tau)\right)}_l- p_l^+(\tau)\Pi_l^+\widetilde{\left(f_-^{2}(0)\right)}_l-q_l^+(\tau)\Pi_l^
+\widetilde{\left(g_-^{2}(0)\right)}_l-
q_l^+(\tau)\Pi_l^+\left(\widetilde{\left(\dot{f}_+^{2}(0)\right)}_l+
\widetilde{\left(\dot{f}_-^{2}(0)\right)}_l\right),\\
\widehat{(\xi_{+}^{1,n})}_l=&\widehat{\left(\Psi_+^{1,n}(\tau)\right)}_l- e^{-i\frac{\delta_l^-\tau}{\eps^2}}\widehat{\left(\Psi_+^{1,n}(0)\right)}_l-
p_l^-(\tau)\Pi_l^+\widetilde{\left(f_+^{1}(0)\right)}_l-
q_l^-(\tau)\Pi_l^+\widetilde{\left(g_+^{1}(0)\right)}_l \\&-q_l^-(\tau)\Pi_l^+\left(\widetilde{\left(\dot{f}_+^{1}(0)\right)}_l+
\widetilde{\left(\dot{f}_-^{1,*}(\tau)\right)}_l\right),\qquad l=-\frac{M}{2},\ldots,\frac{M}{2}-1,\\
\widehat{(\xi_-^{2,n})}_l=&\widehat{\left(\Psi_-^{2,n}(\tau)\right)}_l-
e^{i\frac{\delta_l^-\tau}{\eps^2}}\widehat{\left(\Psi_-^{2,n}(0)\right)}_l
+\overline{p_l^-(\tau)}\Pi_l^-\widetilde{\left(f_-^{2}(0)\right)}_l
+\overline{q_l^-(\tau)}\Pi_l^-\widetilde{\left(g_-^{2}(0)\right)}_l\\&+
\overline{q_l^-(\tau)}\Pi_l^-\left(\widetilde{\left(\dot{f}_+^{2,*}(\tau)\right)}_l
+\widetilde{\left(\dot{f}_-^{2}(0)\right)}_l\right),
\end{cases}
\end{equation}
where
\begin{equation}\label{eq:localerr-def-2}
\begin{cases}
\widehat{\left(\Psi_+^{1,n}\right)}_l(0)=\Pi_l^+\widehat{(\Phi(t_n))}_l,\,\, \widehat{\left(\Psi_-^{1,n}\right)}_l(0)={\bf 0},\,\,\widehat{\left(\Psi_+^{2,n}\right)}_l(0)={\bf 0}, \,\,\widehat{\left(\Psi_-^{1,n}\right)}_l(0)=\Pi_l^-\widehat{(\Phi(t_n))}_l,\\
f_{\pm}^k(s)=W(t_n+s)\Psi_{\pm}^{k,n}(s),\quad \dot{f}_{\pm}^k(s)=W(t_n)\dot{\Psi}_{\pm}^{k,n}(s),\quad g_{\pm}^{k}(s)=\p_tW(t_n) \Psi_{\pm}^{k,n}(s),\\
\dot{f}_{-}^{1,*}(s)=\frac{1}{s}W(t_n)\left(\Psi_{-}^{1,n+1}(s)\right),\quad \dot{f}_{+}^{2,*}(s)=\frac{1}{s}W(t_n)\left(\Psi_{+}^{2,n+1}(s)\right),\quad 0< s\le \tau, \quad k=1,2,
\end{cases}
\end{equation}
and
\be
\dot{\Psi}_{\pm}^{k,n}(s,x)=\sum\limits_{l=-M/2}^{M/2}\widetilde{
\left(\dot\Psi_+^{k,n}(s)\right)}_le^{i\mu_l(x-a)},\quad x\in\Omega, \quad 0\le s\le \tau,\quad k=1,2,
\ee
with
\be\label{eq:localerr-def-3}
\begin{cases}
\widetilde{\left(\dot{\Psi}_+^{1,n}(s)\right)}_l=
-i\frac{2\sin(\mu_l^2\tau/2)}{\delta_l^+\tau}\widetilde{\left(\Psi_+^{1,n}(s)\right)}_l
-i\,\Pi_l^+\widetilde{\left(f_+^1(s)\right)}_l\,,\quad
\widetilde{\left(\dot{\Psi}_-^{1,n}(s)\right)}_l= -i\,\Pi_l^-\widetilde{\left(f_+^{1}(s)\right)}_l\,,\\
\widetilde{\left(\dot{\Psi}_+^{2,n}(s)\right)}_l= -i\,\Pi_l^+\widetilde{\left(f_-^{2}(s)\right)}_l\,,\quad
\widetilde{\left(\dot{\Psi}_-^{2,n}(s)\right)}_l=i\frac{2\sin(\mu_l^2\tau/2)}
{\delta_l^+\tau}\widetilde{\left(\Psi_-^{2,n}(s)\right)}_l-i\,\Pi_l^-\widetilde{\left(f_-^{2}(s)\right)}_l\,.
\end{cases}
\ee
We have the following estimates for the above local truncation error.

\begin{lemma} \label{lem:localerr} Under the assumptions (A) and (B), there
exist constants $0<\tau_0,h_0\leq1$ sufficiently small and independent of $\eps$,
such that for any $0<\eps\leq1$, when $0<\tau\leq\tau_0$ and $0<h\leq  h_0$,
we have the error estimates for the local truncation error
$\xi_{\pm}^{k,n}\in Y_M$ in \eqref{eq:localerr-def}
\begin{equation}\label{eq:localerr}
\|\xi_{\pm}^{k,n}(\cdot)\|_{L^2}\lesssim \tau\left(h^{m_0}+\frac{\tau^2}{\eps^2}\right),\quad
\|\xi_{\pm}^{k,n}(\cdot)\|_{L^2}\lesssim \tau\left(h^{m_0}+\tau^2+\eps^2\right),
\quad 0\le n<\frac{T}{\tau}, \quad k=1,2.
\end{equation}
\end{lemma}

\begin{proof} We will only prove the estimates \eqref{eq:localerr} for $k=1$,
as the proof for $k=2$ is the same.  Using the fact $\delta_l^+\ge1$
and the definitions of $p_l^{\pm}(\tau)$ and $q_l^{\pm}(\tau)$ ($l=-M/2,\ldots,M/2-1$),  we have
\be\label{eq:coefbd}
|p_l^{\pm}(\tau)|\lesssim \tau,\quad |q_l^{\pm}(\tau)|\lesssim\tau^2,\quad |p_l^+(\tau)|\lesssim \eps^2,\quad |q_l^+(\tau)|\lesssim \tau\eps^2.
\ee
Multiplying both sides of the equations in the system \eqref{eq:decsys:1}
by $e^{-i\mu_l(x-a)}$ and integrating over $\Omega$, we
easily recover the equations for $\widehat{(\Psi_{\pm}^{k,n})}_l(s)$, which are exactly the same as  \eqref{eq:fc:1}-\eqref{eq:fc:2} with $\Psi_{\pm,M}^{k,n}$ being replaced
by $\Psi_{\pm}^{k,n}$.

 Following the derivation of the MTI-FP method, it is easy to find that
 the local truncation error comes from the approximations in the integrals
 \eqref{eq:inte1},
\eqref{eq:inte2}, \eqref{eq:inte3} and \eqref{eq:inte4}. In particular,
for  $l=-M/2,\ldots,M/2-1$, we have
\begin{align}\label{eq:error-1}
\widehat{(\xi_-^{1,n})}_l=&-i\int_{0}^{\tau}e^{i\delta_l^+(\tau-s)/\eps^2}
\;\Pi_l^-\left(\widehat{\left(f^{1}_+(s)\right)}_l+\widehat{\left(f^{1}_-(s)\right)}_l\right)\,ds+
\overline{p_l^+(\tau)}\;\Pi_l^-\widetilde{\left(f_+^{1}(0)\right)}_l\\
&+\overline{q_l^+(\tau)}\;\Pi_l^-\widetilde{\left(g_+^{1}(0)\right)}_l
+\overline{q_l^+(\tau)}\;\Pi_l^-\left(\widetilde{\left(\dot{f}_+^{1}(0)\right)}_l+
\widetilde{\left(\dot{f}_-^{1}(0)\right)}_l\right),\nonumber\\
\label{eq:error-2}\widehat{(\xi_{+}^{1,n})}_l=&-i\int_0^{\tau}e^{-i\delta_l^-
(\tau-s)}\;\Pi_l^+\left(\widehat{\left(f^{1}_+(s)\right)}_l+\widehat{\left(f^{1}_-(s)\right)}_l\right)\,ds -p_l^-(\tau)\;\Pi_l^+\widetilde{\left(f_+^{1}(0)\right)}_l\\&-
q_l^-(\tau)\;\Pi_l^+\widetilde{\left(g_+^{1}(0)\right)}_l-
q_l^-(\tau)\;\Pi_l^+\left(\widetilde{\left(\dot{f}_+^{1}(0)\right)}_l+
\widetilde{\left(\dot{f}_-^{1,*}(\tau)\right)}_l\right).\nonumber
\end{align}

\noindent{\it Type I estimate}. Here
we prove the first kind estimate in \eqref{eq:localerr}. Using Taylor expansion, we have
\begin{align}\label{eq:tl:1}
\widehat{(\xi_-^{1,n})}_l=&-i\int_{0}^{\tau}\int_{0}^s\int_0^{s_1}e^{i\delta_l^+
(\tau-s)/\eps^2}\Pi_l^-\left(\widehat{\left(\p_{s_2s_2}f^{1}_+(s_2)\right)}_l+
\widehat{\left(\p_{s_2s_2}
f^{1}_-(s_2)\right)}_l\right)\,ds_2ds_1ds+\widehat{(\eta_{-}^1)}_l,\\
\widehat{(\xi_+^{1,n})}_l=&-i\int_{0}^{\tau}\int_{0}^s\int_0^{s_1}e^{-i\delta_l^-
(\tau-s)/\eps^2}\Pi_l^+\left(\widehat{\left(\p_{s_2s_2}f^{1}_+(s_2)\right)}_l+
\widehat{\left(\p_{s_2s_2}
f^{1}_-(s_2)\right)}_l\right)\,ds_2ds_1ds+\widehat{(\eta_{+}^1)}_l,\label{eq:tl:2}
\end{align}
where $\eta_{\pm}^1(x)=\sum\limits_{l=-M/2}^{M/2-1}\widehat{(\eta_{\pm}^1)}_le^{i\mu_l(x-a)}$ with
\begin{align*}
\widehat{(\eta_{-}^1)}_l=&\overline{p_l^+(\tau)}\;\Pi_l^-\left(-
\widehat{\left(f_+^1(0)\right)}_l+\widetilde{\left(f_+^{1}(0)\right)}_l\right)
+\overline{q_l^+(\tau)}\;\Pi_l^-\left(-\widehat{\left(g_+^{1}(0)\right)}_l+
\widetilde{\left(g_+^{1}(0)\right)}_l\right)\\
&+\overline{q_l^+(\tau)}\;\Pi_l^-\left(-\widehat{\left(\dot{f}_+^{1,n}(0)\right)}_l+
\widetilde{\left(\dot{f}_+^{1}(0)\right)}_l
-\widehat{\left(\dot{f}_-^{1,n}(0)\right)}_l+\widetilde{\left(\dot{f}_-^{1}(0)\right)}_l\right),\\
\widehat{(\eta_{+}^1)}_l=&p_l^-(\tau)\;\Pi_l^+\left(\widehat{\left(f_+^1(0)
\right)}_l-\widetilde{\left(f_+^{1}(0)\right)}_l\right)
+q_l^-(\tau)\;\Pi_l^+\left(\widehat{\left(g_+^{1}(0)\right)}_l-
\widetilde{\left(g_+^{1}(0)\right)}_l\right)\\
&+q_l^-(\tau)\;\Pi_l^+\left(\widehat{\left(\dot{f}_+^{1,n}(0)\right)}_l-
\widetilde{\left(\dot{f}_+^{1}(0)\right)}_l
+\widehat{\left(\dot{f}_-^{1,n}(0)\right)}_l-\widetilde{\left(\dot{f}_-^{1,*}(\tau)\right)}_l\right).
\end{align*}
Here  $\dot{f}_{\pm}^{1,n}(s)$ is given in \eqref{eq:frldef} with $\Psi_{\pm,M}^{k,n}$ being replaced
by $\Psi_{\pm}^{k,n}$. Since $\|\Pi_{l}^{\pm}\|_{l^2}\leq 1$ ($l=-\frac{M}{2},\ldots,\frac{M}{2}-1$) with $\|Q\|_{l^2}$ being the standard $l^2$ norm of the matrix $Q$, using \eqref{eq:coefbd} and triangle inequality, we obtain
\begin{align*}
|\widehat{(\eta_{-}^1)}_l|\lesssim &\tau\left|\widehat{\left(f_+^1(0)\right)}_l-\widetilde{\left(f_+^{1}(0)\right)}_l\right|
+\tau^2\left|\widehat{\left(g_+^{1}(0)\right)}_l-\widetilde{\left(g_+^{1}(0)\right)}_l\right|
+\tau^2\left|\widehat{\left(\dot{f}_+^{1,n}(0)\right)}_l-\widetilde{\left(\dot{f}_+^{1,n}(0)\right)}_l\right|\\
&+\tau^2\left|\widetilde{\left(\dot{f}_+^{1,n}(0)\right)}_l-\widetilde{\left(\dot{f}_+^{1}(0)\right)}_l\right|
+\tau^2\left|\widehat{\left(\dot{f}_+^{1,n}(0)\right)}_l-\widetilde{\left(\dot{f}_+^{1,n}(0)\right)}_l\right|
+\tau^2\left|\widetilde{\left(\dot{f}_-^{1,n}(0)\right)}_l-\widetilde{\left(\dot{f}_-^{1}(0)\right)}_l\right|.
\end{align*}
From the  Parseval's theorem, we get
\begin{align}\label{eq:eta-2}
\|\eta_{-}^1(\cdot)\|_{L^2}^2\lesssim &
\tau^2\|P_M(f_+^1(0))-I_M(f_+^1(0))\|_{L^2}^2+\tau^4\|P_M(g_+^1(0))-I_M(g_+^1(0))\|_{L^2}^2\nonumber\\
&+\tau^4\|P_M(\dot{f}_+^{1,n}(0))-I_M(\dot{f}_+^{1,n}(0))\|_{L^2}^2+
\tau^4\|P_M(\dot{f}_-^{1,n}(0))-I_M(\dot{f}_-^{1,n}(0))\|_{L^2}^2\nonumber\\
&+\tau^4\|I_M(\dot{f}_+^{1,n}(0))-I_M(\dot{f}_+^{1}(0))\|^2_{L^2}+
\tau^4\|I_M(\dot{f}_-^{1,n}(0))-I_M(\dot{f}_-^{1}(0))\|_{L^2}^2.
\end{align}
Recalling assumptions (A) and (B) and noticing Lemma \ref{lem:msregularity}, we have
\begin{align*}
&f_{\pm}^1(0)=W(t_n)\Psi_{\pm}^{1,n}(0)\in H_{p}^{m_0},\quad
g_{\pm}^1(0)=\p_tW(t_n)\Psi_{\pm}^{1,n}(0)\in H_{p}^{m_0},\\
&\dot{f}_{\pm}^{1,n}(0)=W(t_n)\p_s\Psi_{\pm}^{1,n}(0)\in H_p^{m_0-2}.
\end{align*}
Employing \eqref{eq:normeq} and Cauchy inequality, for $m_0\ge4$,
we can bound $\|\eta_{-}^1(\cdot)\|_{L^2}$ from \eqref{eq:eta-2} as
\begin{align}\label{eq:eta_2-1}
\|\eta_{-}^1(\cdot)\|_{L^2}\lesssim &\tau h^{m_0}+\tau^2(h^{m_0}+h^{m_0-2})
+\tau^2\sqrt{h\sum\limits_{j=0}^{M-1}
\left|W(t_n,x_j)(\p_s\Psi_+^{1,n}(0,x_j)-\dot{\Psi}_+^{1,n}(0,x_j))\right|^2}\nonumber\\
&+\tau^2\sqrt{h\sum\limits_{j=0}^{M-1}\left|W(t_n,x_j)(\p_s\Psi_-^{1,n}(0,x_j)-
\dot{\Psi}_-^{1,n}(0,x_j))\right|^2}\nonumber\\
\lesssim &\tau(h^{m_0}+\tau^2)+\tau^2\left(\|I_M(\p_s\Psi_+^{1,n}(0))-\dot{\Psi}_+^{1,n}(0)\|_{L^2}+
\|I_M(\p_s\Psi_-^{1,n}(0))-\dot{\Psi}_-^{1,n}(0)\|_{L^2}\right)\nonumber\\
\lesssim&\tau(h^{m_0}+\tau^2)+\tau^2\left(\|P_M(\p_s\Psi_+^{1,n}(0))-\dot{\Psi}_+^{1,n}(0)\|_{L^2}+
\|P_M(\p_s\Psi_-^{1,n}(0))-\dot{\Psi}_-^{1,n}(0)\|_{L^2}\right).
\end{align}
Using the equation \eqref{eq:decsys:FS:1}, we get
\begin{align*}
\widehat{(\p_s\Psi_+^{1,n})}_l(0)-\widetilde{(\dot{\Psi}_+^{1,n})}_l(0)=&
-i\frac{2\sin(\mu_l^2\tau/2)}{\delta_l^+\tau}\left(\widehat{(\Psi_+^{1,n}(0))}_l
-\widetilde{(\Psi_+^{1,n}(0))}_l\right)\\
&-i\,\Pi_l^+\left(\widehat{\left(f_+^1(0)\right)}_l-\widetilde{(f_+^1(0))}_l\right)
-i\,\left(\delta_l^--\frac{2\sin(\mu_l^2\tau/2)}{\delta_l^+\tau}\right)
\widehat{\left(\Psi_+^{1,n}(0)\right)}_l,\\
\widehat{(\p_s\Psi_-^{1,n})}_l(0)-\widetilde{(\dot{\Psi}_-^{1,n})}_l(0)=
&-i\,\Pi_l^-\left(\widehat{\left(f_+^1(0)\right)}_l-\widetilde{(f_+^1(0))_l}\right),
\end{align*}
and
\be
\|P_M(\p_s\Psi_-^{1,n}(0))-I_M(\dot{\Psi}_-^{1,n}(0))\|_{L^2}\leq
\|P_M(f_+^1(0))-I_M(f_+^1(0))\|_{L^2}\lesssim h^{m_0}.
\ee
Noticing that $|\sin(s)-s|\leq \frac{s^2}{2}$ ($s\in\Bbb R$), we have
\begin{align*}
\left|\delta_l^--\frac{2\sin(\mu_l^2\tau/2)}{\delta_l^+\tau}\right|=
\frac{2}{\delta_l^+}\left|\frac12\mu_l^2-\frac{\sin(\mu_l^2\tau/2)}{\tau}\right|
\lesssim \mu_l^4\tau,\quad l=-\frac{M}{2},\ldots,\frac{M}{2}-1,
\end{align*}
which leads to
\begin{align*}
\left|\widehat{(\p_s\Psi_+^{1,n})}_l(0)-\widetilde{(\dot{\Psi}_+^{1,n})}(0)\right|\lesssim&
\frac{1}{\tau}\left|\widehat{(\Psi_+^{1,n}(0))}_l
-\widetilde{(\Psi_+^{1,n}(0))}_l\right|
+\left|\widehat{\left(f_+^1(0)\right)}_l-\widetilde{(f_+^1(0))}_l\right|
+\tau\mu_l^4\left|\widehat{(\Psi_+^{1,n}(0))}_l\right|,
\end{align*}
and for $m_0\ge4$,
\begin{align*}
&\|P_M(\p_s\Psi_+^{1,n}(0))-I_M(\dot{\Psi}_+^{1,n}(0))\|_{L^2}\\
&\ \lesssim\frac{1}{\tau}\|P_M(\Psi_+^{1,n}(0))-I_M(\Psi_+^{1,n}(0))\|_{L^2}
+\tau\|P_M(\Psi_+^{1,n}(0))\|_{H^4}
+\|P_M(f_+^1(0))-I_M(f_+^1(0))\|_{L^2}\\
&\ \lesssim h^{m_0}+\tau+h^{m_0}/\tau.
\end{align*}
Combining the above estimates with \eqref{eq:eta_2-1}, we derive
\be\label{eq:eta-err}
\|\eta_-^{1}(\cdot)\|_{L^2}\lesssim \tau(h^{m_0}+\tau^2)+\tau^2(h^{m_0}+h^{m_0}/\tau+\tau)
\lesssim \tau(h^{m_0}+\tau^2).
\ee
By the same procedure, $\|\eta_+^{1}(\cdot)\|_{L^2}$ can be bounded as
\begin{align*}
\|\eta_+^1(\cdot)\|_{L^2}\lesssim&\tau(\tau^2+h^{m_0})+
\tau^2\|P_M(\p_s\Psi_-^{1,n}(0))-P_M(\Psi_-^{1,n}(\tau))/\tau\|_{L^2}\\
&+\tau\|P_M(\Psi_-^{1,n}(\tau))-I_M(\Psi_-^{1,n}(\tau))\|_{L^2},
\end{align*}
where Taylor expansion gives
\begin{align*}
\p_s\Psi_-^{1,n}(0)-\Psi_-^{1,n}(\tau)/\tau=-\tau\int_0^1\int_0^s\p_{ss}\Psi_{-}^{1,n}(s_1\tau)\,ds_1ds.
\end{align*}
Thus, recalling Lemma \ref{lem:msregularity}, we estimate
\be\label{eq:eta-err2}
\|\eta_+^1(\cdot)\|_{L^2}\lesssim\tau(\tau^2+h^{m_0})+\tau^3
\|\p_{ss}\Psi_{-}^{1,n}(\cdot)\|_{L^\infty([0,\tau];(L^2)^2)}
\lesssim \tau\left(h^{m_0}+\frac{\tau^2}{\eps^2}\right).
\ee
Now, Lemma \ref{lem:msregularity} together with \eqref{eq:tl:1},
\eqref{eq:tl:2}, \eqref{eq:eta-2}, \eqref{eq:eta_2-1}
and \eqref{eq:eta-err2} implies
\begin{align}
\|\xi_{\pm}^{1,n}(\cdot)\|_{L^2}\lesssim &\,\tau^3\left(\|\p_{ss}(W(t_n+s)\Psi_{+}^{1,n}(s))\|_{L^\infty([0,\tau];(L^2)^2)}+
\|\p_{ss}(W(t_n+s)\Psi_{-}^{1,n}(s))\|_{L^\infty([0,\tau];(L^2)^2)}\right)\nonumber\\
&+\|\eta_{\pm}^{1}(\cdot)\|_{L^2}
\lesssim \tau\left(h^{m_0}+\frac{\tau^2}{\eps^2}\right).\label{eq:xierro-1}
\end{align}

\noindent{\it Type II estimate}. Now we prove the second estimate
 for $\xi_{\pm}^{1,n}(x)$ in \eqref{eq:localerr}. Starting from \eqref{eq:error-1}
and \eqref{eq:error-2}, we treat the terms involving $f_{+}^{1}(s)$,
$\dot{f}_+^{1}(s)$ and $g_+^1(s)$ in the same way as in proving \eqref{eq:xierro-1},
and leave the rest terms as
\begin{align*}
\widehat{(\xi_-^{1,n})}_l=&\widehat{(\zeta_-^{1})}_l-i\int_{0}^{\tau}
e^{i\delta_l^+(\tau-s)/\eps^2}\;\Pi_l^-\widehat{\left(f^{1}_-(s)\right)}_l\,ds
+\overline{q_l^+(\tau)}\;\Pi_l^-\widetilde{\left(\dot{f}_-^{1}(0)\right)}_l,\nonumber\\
\widehat{(\xi_{+}^{1,n})}_l=&\widehat{(\zeta_{+}^{1})}_l-i\int_0^{\tau}
e^{-i\delta_l^-(\tau-s)}\;\Pi_l^+\widehat{\left(f^{1}_-(s)\right)}_l\,ds-
q_l^-(\tau)\;\Pi_l^+\widetilde{\left(\dot{f}_-^{1,*}(\tau)\right)}_l,\nonumber
\end{align*}
with $\zeta_{\pm}^1(x)=\sum\limits_{l=-M/2}^{M/2-1}
\widehat{(\zeta_{\pm}^1)}_le^{i\mu_l(x-a)}$ satisfying
\begin{equation*}
\|\zeta_{\pm}^{1}(\cdot)\|_{L^2}\lesssim\tau(h^{m_0}+\tau^2).
\end{equation*}
The proof of the above decomposition and the corresponding error
bounds for $\zeta_{\pm}^1(x)$ is identical to the proof of \eqref{eq:xierro-1}, and
thus is omitted here for brevity. Applying triangle inequality and \eqref{eq:coefbd}, we have
\begin{align*}
\left|\widehat{(\xi_-^{1,n})}_l\right|\leq& \left|
\widehat{(\zeta_-^{1})}_l\right|+\int_{0}^{\tau}\left|
\widehat{\left(f^{1}_-(s)\right)}_l\right|\,ds
+\tau\eps^2\left|\widetilde{\left(\dot{f}_-^{1}(0)\right)}_l\right|,\\
\left|\widehat{(\xi_{+}^{1,n})}_l\right|\leq&\left|\widehat{(\zeta_{+}^{1})}_l
\right|+\int_0^{\tau}\left|\widehat{\left(f^{1}_-(s)\right)}_l\right|\,ds+
\tau^2\left|\widetilde{\left(\dot{f}_-^{1,*}(\tau)\right)}_l\right|.
\end{align*}
Recalling Lemma \ref{lem:msregularity} which implies
$\|\Psi_-^{1,n}(s)\|_{H^{m_0-1}_p}\lesssim\eps^2$,
we  know $\|f_-^1(s)\|_{L^2}\lesssim \eps^2$,
$\|\dot{f}_-^{1,*}(\tau)\|_{H_p^{m_0-1}}\lesssim\eps^2/\tau$ and
\begin{align*}
\|\dot{f}^1_-(0)\|_{H^{m_0-1}}\lesssim\|\dot{\Psi}^{1,n}_-(0)\|_{H^{m_0-1}}
\lesssim \|I_M(f_{+}^{1}(0))\|_{H^{m_0-1}}\lesssim\|f_{+}^{1}(0)\|_{H^{m_0}}\lesssim1.
\end{align*}
Hence, using the Parseval's theorem, we get
 \begin{align*}
\left\|\xi_-^{1,n}(\cdot)\right\|_{L^2}\lesssim&
\left\|\zeta_-^{1}(\cdot)\right\|_{L^2}+\tau\left\|f^{1}_-(\cdot)
\right\|_{L^\infty([0,\tau];(L^2)^2)}
+\tau\eps^2\left\|I_M(\dot{f}_-^{1}(0))\right\|_{L^2}
\lesssim\tau(h^{m_0}+\tau^2+\eps^2),\\
\left\|\xi_{+}^{1,n}(\cdot)\right\|_{L^2}\lesssim&\left\|\zeta_{+}^{1}(\cdot)
\right\|_{L^2}+\tau\left\|f^{1}_-(\cdot)\right\|_{L^\infty([0,\tau];(L^2)^2)}
+\tau^2\left\|I_M(\dot{f}_-^{1,*}(\tau))\right\|_{L^2}\lesssim \tau(h^{m_0}+\tau^2+\eps^2),
\end{align*}
which, together with \eqref{eq:xierro-1}, completes the proof for \eqref{eq:localerr}.
\end{proof}

\smallskip

Subtracting  \eqref{eq:scheme:3} from \eqref{eq:localerr-def},
noticing \eqref{eq:scheme:4} and \eqref{eq:localerr-def-2},
we get error equations for $\bez_{\pm}^{k,n+1}(x)$ ($k=1,2$) in \eqref{eq:ez-def} as
\begin{equation}\label{eq:erroreq}
\begin{cases}
\widetilde{(\bez_-^{1,n+1})}_l=\widetilde{(\calF^{1,n}_-)}_l+\widehat{(\xi_-^{1,n})}_l\,,\quad
\widetilde{(\bez_{+}^{1,n+1})}_l=e^{-i\frac{\delta_l^-\tau}{\eps^2}}\Pi_l^+
\widetilde{(\bee^n)}_l+\widetilde{(\calF^{1,n}_+)}_l+\widehat{(\xi_+^{1,n})}_l\,,\\
\widetilde{(\bez_+^{2,n+1})}_l=\widetilde{(\calF^{2,n}_+)}_l
+\widehat{(\xi_+^{2,n})}_l\,,\quad
\widetilde{(\bez_{-}^{2,n+1})}_l=e^{i\frac{\delta_l^-\tau}{\eps^2}}\Pi_l^-
\widetilde{(\bee^n)}_l+\widetilde{(\calF^{2,n}_-)}_l+\widehat{(\xi_-^{2,n})}_l\,,
\end{cases}
\end{equation}
where $\calF^{k,n}_{\pm}(x)=\sum\limits_{l=-M/2}^{M/2-1}
\widetilde{(\calF^{k,n}_{\pm})}_le^{i\mu_l(x-a)}$ ($k=1,2$) is given by
\be\label{eq:Fdef}
\begin{cases}
\widetilde{\left(\calF_-^{1,n}\right)}_l&=-\overline{p_l^+(\tau)}\Pi_l^-
\widetilde{\left(F_+^{1,n}\right)}_l
-\overline{q_l^+(\tau)}\Pi_l^-\widetilde{\left(G_+^{1,n}\right)}_l
-\overline{q_l^+(\tau)}\Pi_l^-\left(\widetilde{\left(\dot{F}_+^{1,n}\right)}_l+
\widetilde{\left(\dot{F}_-^{1,n}\right)}_l\right),\\
\widetilde{\left(\calF_+^{2,n}\right)}_l&=p_l^+(\tau)\Pi_l^+
\widetilde{\left(F_-^{2,n}\right)}_l+
q_l^+(\tau)\Pi_l^+\widetilde{\left(G_-^{2,n}\right)}_l
+q_l^+(\tau)\Pi_l^+\left(\widetilde{\left(\dot{F}_+^{2,n}\right)}_l+
\widetilde{\left(\dot{F}_-^{2,n}\right)}_l\right),\\
\widetilde{\left(\calF_+^{1,n}\right)}_l&=p_l^-(\tau)\Pi_l^+
\widetilde{\left(F_+^{1,n}\right)}_l+
q_l^-(\tau)\Pi_l^+\widetilde{\left(G_+^{1,n}\right)}_l +q_l^-(\tau)\Pi_l^+\left(\widetilde{\left(\dot{F}_+^{1,n}\right)}_l+
\widetilde{\left(\dot{F}_-^{1,*}\right)}_l\right),\\
\widetilde{\left(\calF_-^{2,n}\right)}_l&= -\overline{p_l^-(\tau)}\Pi_l^-
\widetilde{\left(F_-^{2,n}\right)}_l
-\overline{q_l^-(\tau)}\Pi_l^-\widetilde{\left(G_-^{2,n}\right)}_l-
\overline{q_l^-(\tau)}\Pi_l^-\left(\widetilde{\left(\dot{F}_+^{2,*}\right)}_l+
\widetilde{\left(\dot{F}_-^{2,n}\right)}_l\right),
\end{cases}
\ee
with  $\dot{F}_{\pm}^{k,n}(x)=\sum\limits_{l=-M/2}^{M/2-1}
\widetilde{(\dot{F}^{k,n}_{\pm})}_le^{i\mu_l(x-a)}\in Y_M$ ($k=1,2$),
$G_{\pm}^{k_\pm,n}(x)=\sum\limits_{l=-M/2}^{M/2-1}
\widetilde{(G^{k_\pm,n}_{\pm})}_le^{i\mu_l(x-a)}\in Y_M$
$F_{\pm}^{k_{\pm},n}(x)=\sum\limits_{l=-M/2}^{M/2-1}
\widetilde{(F^{k_\pm,n}_{\pm})}_le^{i\mu_l(x-a)}\in Y_M$,
 and $\dot{F}_{\pm}^{k_{\mp},*}(x)=\sum\limits_{l=-M/2}^{M/2-1}
 \widetilde{(\dot{F}^{k_{\mp},*}_{\pm})}_le^{i\mu_l(x-a)}\in Y_M$ ($k_+=1,k_-=2$) defined as
\be\label{eq:Fdef2}
\begin{split}
&\widetilde{(F^{k_{\pm},n}_{\pm})}_l=
\widetilde{(f_{\pm}^{k_{\pm}}(0))}_l-\widetilde{(f_{\pm}^{k_{\pm}})}_l,
\quad \widetilde{(\dot{F}^{k,n}_{\pm})}_l=
\widetilde{(\dot{f}_{\pm}^{k}(0))}_l-\widetilde{(\dot{f}_{\pm}^k)}_l,\quad
\widetilde{(G^{k_\pm,n}_{\pm})}_l=
\widetilde{(g_{\pm}^{k_\pm}(0))}_l-\widetilde{(g_{\pm}^{k_\pm})}_l,\\
&\widetilde{(\dot{F}^{1,*}_{-})}_l=
\widetilde{(f_{-}^{1,*}(0))}_l-\widetilde{(f_{-}^{1,*})}_l,\quad
\widetilde{(\dot{F}^{2,*}_{+})}_l=
\widetilde{(f_{+}^{2,*}(0))}_l-\widetilde{(f_{+}^{2,*})}_l,\quad l=-\frac{M}{2},\ldots,\frac{M}{2}-1.
\end{split}
\ee
For the electromagnetic error part
$\calF_{\pm}^{k,n}(x)$ ($k=1,2$, $0\leq n\leq \frac{T}{\tau}-1$), we have the lemma below.
\begin{lemma}\label{lem:F} Under the assumptions (A) and (B),
the electromagnetic error part $\calF_{\pm}^{k,n}(x)\in Y_M$ ($k=1,2$,
$0\leq n\leq\frac{T}{\tau}-1$) defined in \eqref{eq:Fdef} with \eqref{eq:Fdef2} satisfies
the bounds as
\begin{align*}
&\|F_{\pm}^{k_{\pm},n}(\cdot)\|_{L^2}+\|G^{k_{\pm},n}_{\pm}(\cdot)\|_{L^2}+
\|\dot{F}_{\pm}^{3-k_{\pm},n}(\cdot)\|_{L^2}\lesssim h^{m_0}+
\|\bee^n(\cdot)\|_{L^2},\quad k_+=1,\quad k_-=2,\\
&
\|\dot{F}_+^{1,n}(\cdot)\|_{L^2}+\|\dot{F}_-^{2,n}(\cdot)\|_{L^2}\lesssim \frac{1}{\tau}(h^{m_0}+\|\bee^n(\cdot)\|_{L^2}),
\quad \|\dot{F}_{\pm}^{k_{\mp},*}(\cdot)\|_{L^2}\lesssim \frac{1}{\tau}(h^{m_0}+\|\bez_{\pm}^{k_{\mp},n+1}(\cdot)\|_{L^2}),
\end{align*}
which implies that
\begin{align}\label{eq:Fbound}
&\|\calF_{\pm}^{k_{\pm},n}(\cdot)\|_{L^2}\lesssim \tau(h^{m_0}+\|\bez_\mp^{k_{\pm},n+1}(\cdot)\|_{L^2}+\|\bee^n(\cdot)\|_{L^2}),\quad
\|\calF_{\pm}^{k_{\mp},n}(\cdot)\|_{L^2}\lesssim \tau(h^{m_0}+\|\bee^n(\cdot)\|_{L^2}).
\end{align}
\end{lemma}

\begin{proof}Recalling the assumptions (A) and (B), Lemma \ref{lem:msregularity}, \eqref{eq:Fdef2}, \eqref{eq:localerr-def-2}, \eqref{eq:localerr-def-3}, \eqref{eq:scheme:4} and
\eqref{eq:scheme:5}, applying the Parseval's theorem,
we have
\begin{align*}
\|F^{1,n}_{+}(\cdot)\|_{L^2}^2\leq&\|I_M(f_{+}^1(0))-I_M(f_{+}^1)\|_{L^2}^2
=h\sum\limits_{j=0}^{M-1}\left|W(t_n,x_j)(\Psi_{+}^{1,n}(0,x_j)-\Psi_{+,j}^1)\right|^2\\
\lesssim& h\sum\limits_{j=0}^{M-1}\left|\Psi_{+}^{1,n}(0,x_j)-\Psi_{+,j}^1\right|^2\lesssim h^{2m_0}+\|P_M(\Phi(t_n))-I_M(\Phi^n)\|_{L^2}^2
\lesssim h^{2m_0}+\|\bee^n(\cdot)\|_{L^2}^2,
\end{align*}
and  similarly we have $\|F^{2,n}_{-}(\cdot)\|_{L^2}\lesssim h^{m_0}+\|\bee^n(\cdot)\|_{L^2}$.
Using the same idea, we can obtain
\begin{equation*}
\|G^{1,n}_{+}(\cdot)\|_{L^2}+\|G_-^{2,n}(\cdot)\|_{L^2}\lesssim h^{m_0}+\|\bee^n(\cdot)\|_{L^2},
\end{equation*}
and
\begin{align*}
\|F^{1,*}_{-}(\cdot)\|_{L^2}\lesssim &\|I_M(f_-^{1,*}(\tau))-I_M(f_-^{1,*})\|_{L^2}
\lesssim \frac{1}{\tau}\|I_M(\Psi_-^{1,n}(\tau))-I_M(\Psi_-^{1,n+1})\|_{L^2}\\
\lesssim& \frac{1}{\tau}(h^{m_0}+\|\bez_{-}^{1,n+1}(\cdot)\|_{L^2}),\\
\|F^{2,*}_{+}(\cdot)\|_{L^2}\lesssim &\|I_M(f_+^{2,*}(\tau))-I_M(f_+^{2,*})\|_{L^2}
\lesssim\frac{1}{\tau}(h^{m_0}+\|\bez_{+}^{2,n+1}(\cdot)\|_{L^2}).
\end{align*}
It remains to estimate $\|\dot{F}^{k,n}_{\pm}(\cdot)\|_{L^2}$.
Again, in the same spirit of the above arguments, we arrive at
\begin{equation}\label{eq:dotfestimate}
\|\dot{F}^{k,n}_{\pm}(\cdot)\|_{L^2}\lesssim
\|I_M(\dot{\Psi}_{\pm}^{k,n}(0))-I_M(\dot{\Psi}_{\pm}^{k})\|_{L^2}.
\end{equation}
Comparing \eqref{eq:localerr-def-3} with \eqref{eq:scheme:5},
noticing the properties of $\delta_l^{\pm}$ and the arguments in the above proof, we find
\begin{align*}
\|I_M(\dot{\Psi}_{-}^{1,n}(0))-I_M(\dot{\Psi}_{-}^{1})\|_{L^2}
\lesssim& \|I_M(f_+^{1}(0))-I_M(f_+^{1})\|_{L^2}\lesssim h^{m_0}+\|\bee^n(\cdot)\|_{L^2},\\
\|I_M(\dot{\Psi}_{+}^{2,n}(0))-I_M(\dot{\Psi}_{+}^{2})\|_{L^2}
\lesssim &\|I_M(f_-^{2}(0))-I_M(f_-^{2})\|_{L^2}\lesssim h^{m_0}+\|\bee^n(\cdot)\|_{L^2},\\
\|I_M(\dot{\Psi}_{+}^{1,n}(0))-I_M(\dot{\Psi}_{+}^{1})\|_{L^2}
\lesssim&\frac{1}{\tau}\|I_M(\Psi_{+}^{1,n}(0))-
I_M(\Psi_{+}^{1})\|_{L^2}+\|I_M(f_+^{1}(0))-I_M(f_+^{1})\|_{L^2}\\
\lesssim&\frac{1}{\tau}\left(h^{m_0}+
\|\bee^n(\cdot)\|_{L^2}\right)+h^{m_0}+\|\bee^n(\cdot)\|_{L^2},\\
\|I_M(\dot{\Psi}_{-}^{2,n}(0))-I_M(\dot{\Psi}_{-}^{2})\|_{L^2}
\lesssim&\frac{1}{\tau}\|I_M(\Psi_{-}^{2,n}(0))-
I_M(\Psi_{-}^{2})\|_{L^2}+\|I_M(f_-^{2}(0))-I_M(f_-^{2})\|_{L^2}\\
\lesssim&\frac{1}{\tau}\left(h^{m_0}+
\|\bee^n(\cdot)\|_{L^2}\right)+h^{m_0}+\|\bee^n(\cdot)\|_{L^2},
\end{align*}
which implies the bounds for $\|\dot{F}^{k,n}_{\pm}(\cdot)\|_{L^2}$ in
view of \eqref{eq:dotfestimate}. Combining all the above results,
recalling \eqref{eq:Fdef2} and properties of the
coefficients $p^{\pm}_l(\tau)$ and $q^{\pm}_l(\tau)$ in \eqref{eq:coefbd}, we conclude that
\eqref{eq:Fbound} holds.
\end{proof}

\smallskip

Now, we are ready to prove the main theorem.

{\it Proof of Theorem \ref{thm:main}.}
Recalling the decomposition \eqref{eq:dec:error} and the
error equation \eqref{eq:erroreq}, we  get
\begin{align}
\widetilde{(\bee^{n+1})}_l=&e^{-i\tau/\eps^2}\left(
\widetilde{\left(\bez_+^{1,n+1}\right)}_l+\widetilde{\left(\bez_-^{1,n+1}\right)}_l\right)
+e^{i\tau/\eps^2}\left(\widetilde{\left(\bez_+^{2,n+1}\right)}_l+
\widetilde{\left(\bez_-^{2,n+1}\right)}_l\right)\nonumber\\
=&\left(e^{-i\delta_l/\eps^2}\Pi_l^+\widetilde{(\bee^n)}_l+
e^{i\delta_l/\eps^2}\Pi_l^-\widetilde{(\bee^n)}_l\right)+
\widetilde{(\chi^n)_l},\label{eq:egrowth}
\end{align}
with $\chi^n(x)=\sum\limits_{l=-M/2}^{M/2-1}\widetilde{(\chi^n)}_le^{i\mu_l(x-a)}\in Y_M$ given as
\be\label{eq:chi-def}
\widetilde{(\chi^n)}_l=\sum\limits_{k=1,2}e^{i\tau(2k-3)/\eps^2}\left(\widetilde{(\calF_+^{k,n})}_l
+\widetilde{(\calF_-^{k,n})}_l+\widehat{(\xi_+^{k,n})}_l+\widehat{(\xi_-^{k,n})}_l\right),\quad l=-\frac{M}{2},\ldots,\frac{M}{2}-1.
\ee
In particular $\|\bee^{0}(\cdot)\|_{L^2}=\|P_M(\Phi_0)-I_M(\Phi_0)\|_{L^2}\lesssim h^{m_0}$.

Taking the $l^2$ norm of the vectors in the error equation \eqref{eq:erroreq}, then summing together for $l=-M/2,\ldots,M/2-1$, utilizing  Lemma \ref{lem:F} and Parserval's theorem, there holds for $0<\tau\leq1$,
\begin{align*}
\|\bez_-^{1,n+1}(\cdot)\|_{L^2}\lesssim&\|\calF_-^{1,n}(\cdot)\|_{L^2}+\|\xi_-^{1,n}(\cdot)\|_{L^2}\lesssim
\tau(h^{m_0}+\|\bee^n(\cdot)\|_{L^2})+\|\xi_-^{1,n}(\cdot)\|_{L^2},\\
\|\bez_+^{2,n+1}(\cdot)\|_{L^2}\lesssim&\|\calF_+^{2,n}(\cdot)\|_{L^2}+\|\xi_+^{2,n}(\cdot)\|_{L^2}\lesssim
\tau(h^{m_0}+\|\bee^n(\cdot)\|_{L^2})+\|\xi_+^{2,n}(\cdot)\|_{L^2},
\end{align*}
and so
\begin{align*}
\|\chi^n(\cdot)\|\lesssim&\tau(h^{m_0}+\|\bee^n(\cdot)\|_{L^2}+\|\bez_-^{1,n+1}(\cdot)
\|_{L^2}+\|\bez_+^{2,n+1}(\cdot)\|_{L^2})+
\sum\limits_{k=1,2}\left(\|\xi_+^{k,n}(\cdot)\|_{L^2}+\|\xi_-^{k,n}(\cdot)\|_{L^2}\right)\\
\lesssim&\tau(h^{m_0}+\|\bee^n(\cdot)\|_{L^2})+
\sum\limits_{k=1,2}\left(\|\xi_+^{k,n}(\cdot)\|_{L^2}+\|\xi_-^{k,n}(\cdot)\|_{L^2}\right).
\end{align*}
From Lemma \ref{lem:localerr} on the local truncation error $\xi_{\pm}^{k,n}(x)$, we get
\begin{align}
\label{eq:chi:1}\|\chi^n(\cdot)\|_{L^2}\lesssim&
\tau\|\bee^{n}(\cdot)\|_{L^2}+\tau(h^{m_0}+\tau^2/\eps^2),\quad 0\leq n\leq\frac{T}{\tau}-1,\\
\label{eq:chi2:}\|\chi^n(\cdot)\|_{L^2}\lesssim&
\tau\|\bee^{n}(\cdot)\|_{L^2}+\tau(h^{m_0}+\tau^2+\eps^2),\quad 0\leq n\leq\frac{T}{\tau}-1.
\end{align}
Now, taking the $l^2$ norm of the vectors on both sides of \eqref{eq:egrowth},
making use of the orthogonal properties of $\Pi_{l}^{\pm}$ where
$\left|e^{i\theta_1}\Pi_l^+{\bf v}+e^{i\theta_2}\Pi_l^-{\bf v}\right|^2=
|\Pi_l^+{\bf v}|^2+|\Pi_l^-{\bf v}|^2=|{\bf v}|^2$ for all ${\bf v}\in
\mathbb{C}^2,\theta_1,\theta_2\in\Bbb R$,  we
can have
\begin{align*}
\left|\widetilde{(\bee^{n+1})}_l\right|^2=&\left|e^{-i\delta_l/\eps^2}
\Pi_l^+\widetilde{(\bee^n)}_l+e^{i\delta_l/\eps^2}\Pi_l^-
\widetilde{(\bee^n)}_l\right|^2+|\widetilde{(\chi^n)}_l|^2\\
&+2\text{Re}\left((e^{-i\delta_l/\eps^2}\Pi_l^+\widetilde{(\bee^n)}_l+
e^{i\delta_l/\eps^2}\Pi_l^-\widetilde{(\bee^n)}_l)^*\widetilde{(\chi^n)}_l\right)\\
=&|\widetilde{(\bee^n)}_l|^2+|\widetilde{(\chi^n)}_l|^2
+2\text{Re}\left((e^{-i\delta_l/\eps^2}\Pi_l^+\widetilde{(\bee^n)}_l+
e^{i\delta_l/\eps^2}\Pi_l^-\widetilde{(\bee^n)}_l)^*\widetilde{(\chi^n)}_l\right),
\end{align*}
where  $Re(c)$ denotes the real part of the complex number $c$.
Applying Cauchy inequality, we find
\be\label{eq:eg:1}
\left|\widetilde{(\bee^{n+1})}_l\right|^2-\left|\widetilde{(\bee^{n})}_l\right|^2\lesssim \tau|\widetilde{(\bee^n)}_l|^2+\frac{1}{\tau}|\widetilde{(\chi^n)}_l|^2,\quad
l=-\frac{M}{2},\ldots,\frac{M}{2}-1.
\ee
Summing \eqref{eq:eg:1} together for $l=-\frac{M}{2},\ldots,\frac{M}{2}-1$
and using Parseval's theorem, we obtain
\be\label{eq:eg:2}
\|\bee^{n+1}(\cdot)\|_{L^2}^2-\|\bee^{n}(\cdot)\|_{L^2}^2\lesssim\tau\|\bee^{n}
(\cdot)\|_{L^2}^2+\frac{1}{\tau}\|\chi^n(\cdot)\|^2_{L^2},\quad 0\leq n\leq \frac{T}{\tau}-1.
\ee
Summing \eqref{eq:eg:2} for indices  $0,1,\ldots,n$ and using
\eqref{eq:chi:1}-\eqref{eq:chi2:}, we derive that for $ 0\leq n\leq \frac{T}{\tau}-1$,
\begin{align}
\|\bee^{n+1}(\cdot)\|_{L^2}^2-\|\bee^{0}(\cdot)\|_{L^2}^2\lesssim&\tau
\sum\limits_{m=1}^{n}\|\bee^{m}(\cdot)\|_{L^2}^2
+n\tau\left(h^{m_0}+\frac{\tau^2}{\eps^2}\right)^2,\\
\|\bee^{n+1}(\cdot)\|_{L^2}^2-\|\bee^{0}(\cdot)\|_{L^2}^2\lesssim&\tau
\sum\limits_{m=1}^{n}\|\bee^{m}(\cdot)\|_{L^2}^2
+n\tau(h^{m_0}+\tau^2+\eps^2)^2.
\end{align}
Since $\|\bee^0(\cdot)\|_{L^2}\lesssim h^{m_0}$, Gronwall's
inequality will lead to the conclusion when $0<\tau\le \tau_0\le1$ and $0<h\le h_0\le1$ sufficiently small
\be
\|\bee^{n+1}(\cdot)\|_{L^2}^2\lesssim\left(h^{m_0}+\frac{\tau^2}{\eps^2}\right)^2,
\quad\|\bee^{n+1}(\cdot)\|_{L^2}^2\lesssim(h^{m_0}+\tau^2+\eps^2)^2,\quad 0\leq n\leq \frac{T}{\tau}-1.
\ee
In view of \eqref{eq:error-tri}, we  conclude that \eqref{eq:est:main} holds. $\hfill\Box$

\section{Numerical results}\label{sec:num}
In this section, we present numerical tests on our MTI-FP method \eqref{eq:scheme:1}
and apply it to study numerically the
convergence of the Dirac equation \eqref{SDEdd} to its limiting
Schr\"odinger model \eqref{eq:MD:1st} and the second order limiting  Pauli-type equation model.
To this purpose, we choose the electromagnetic potentials in the Dirac equation (\ref{SDEdd}) with
$d=1$ as
\begin{eqnarray}\label{eq:pot:nm}
A_1(t,x) = \frac{(x+1)^2}{1+x^2},\qquad V(t,x) = \frac{1-x}{1+x^2}, \qquad x\in{\mathbb R}, \quad t\ge0,
\end{eqnarray}
and the initial data $\Phi_0(x)=(\phi_1(x),\phi_2(x))^T$ as
\begin{equation}\label{eq:ini:nm}
\phi_1(x) = e^{-x^2/2}, \quad\phi_2(x) = e^{-(x-1)^2/2}, \qquad x\in{\mathbb  R}.
\end{equation}

\subsection{Accuracy test}
The Dirac equation \eqref{SDEdd} with $d=1$, \eqref{eq:pot:nm} and \eqref{eq:ini:nm}
is solved numerically on an interval $\Omega=(-16, 16)$, i.e. $a=-16$ and $b=16$,
with periodic boundary conditions.
The `reference' solution  $\Phi(t,x)=(\phi_1(t, x),\phi_2(t, x))^T$
is obtained numerically by using the TSFP method \cite{BCJ} with a very fine mesh size and a small time step,
e.g. $h_e = 1/32$ and $\tau_e = 10^{-7}$. Denote $\Phi^n_{h,\tau}$ as the numerical solution
obtained by the MTI-FP method with  mesh size $h$ and time step $\tau$.
In order to quantify the convergence, we introduce
\[e_{h,\tau}(t_n)=\|\Phi^n-\Phi(t_n,\cdot)\|_{l^2}=
\sqrt{h\sum_{j=0}^{M-1}|\Phi^n_j-\Phi(t_n,x_j)|^2}.\]

Tab. \ref{table_spatial} displays the spatial errors $e_{h,\tau}(t=2.0)$ with $\tau=10^{-4}$ for
different $\eps$ and $h$; and Tab. \ref{table_temp} lists the temporal errors $e_{h,\tau}(t=2.0)$
with $h=1/32$ for different $\eps$ and $\tau$. From Tabs. \ref{table_spatial}-\ref{table_temp} and additional
numerical results  not shown here for brevity, we can
draw the following conclusions for the MTI-FP method:

\begin{table}[t!]
\def\temptablewidth{1\textwidth}
\vspace{-12pt}
\caption{Spatial error analysis of
the MTI-FP method for the Dirac equation in 1D. }
{\rule{\temptablewidth}{1pt}}
\begin{tabular*}{\temptablewidth}{@{\extracolsep{\fill}}cccccc}
$e_{h,\tau}(2.0)$ & $h_0=2$   & $h_0$/2   &$h_0/2^2$     &  $h_0/2^3$ & $h_0/2^4$  \\
\hline
$\varepsilon_0=1$ &  1.65  &  5.74E-1  & 7.08E-2  &  7.00E-5 & 8.53E-10  \\
$\varepsilon_0/2$ &  1.39  &  3.45E-1  & 7.06E-3  &  6.67E-6 & 9.71E-10 \\
$\varepsilon_0/2^2$ &  1.18  &  1.67E-1  & 1.71E-3  &  1.43E-6 & 1.10E-9 \\
$\varepsilon_0/2^3$ &  1.13  &  1.46E-1  & 1.03E-3  &  6.77E-7 & 9.16E-10 \\
$\varepsilon_0/2^4$ &  1.15  &  1.45E-1  & 8.52E-4  &  4.86E-7 & 1.33E-9 \\
\end{tabular*}
{\rule{\temptablewidth}{1pt}} \label{table_spatial}
\end{table}

\begin{table}[t!]
\def\temptablewidth{1\textwidth}
\vspace{-12pt}
\caption{Temporal error analysis of
the MTI-FP method for the Dirac equation in 1D. The convergence order is calculated as $\log_2(e_{h,2\tau}/e_{h,\tau})$.}
\begin{tabular}{ c c c c c c c c c c c }
\toprule
\hline
$e_{h,\tau}(2.0)$ & $\tau_0=0.1$ & $\tau_0/2$ & $\tau_0/2^2$
& $\tau_0/2^3$ & $\tau_0/2^4$ & $\tau_0/2^5$ & $\tau_0/2^6$ & $\tau_0/2^7$ & $\tau_0/2^8$  \\
\hline
$\eps_0=1$ & 3.69E-2 & 9.18E-3 & 2.29E-3 & 5.73E-4
& 1.43E-4 & 3.58E-5 & 8.94E-6 & 2.24E-6 & 5.59E-7  \\
order&-&2.01&2.00&2.00&2.00&2.00&2.00&2.00&2.00\\
\hline
$\eps_0/2$ & 5.98E-2 & 1.51E-2 & 3.77E-3 & 9.45E-4
& 2.36E-4 & 5.90E-5 & 1.48E-5 & 3.69E-6 & 9.23E-7 \\
order&-&1.99&2.00&2.00&2.00&2.00&2.00&2.00&2.00\\ \hline
$\eps_0/2^2$ & \underline{1.91E-1} & 5.67E-2 & 1.47E-2 & 3.74E-3
& 9.39E-4 & 2.35E-4 & 5.87E-5 & 1.47E-5 & 3.67E-6  \\
order&-&1.75&1.95&1.97& 1.99& 2.00& 2.00& 2.00&    2.00\\ \hline
$\eps_0/2^3$ & 7.12E-2 & \underline{7.17E-2} & \underline{4.90E-2}
& 1.48E-2 & 3.89E-3 & 9.84E-4 & 2.47E-4 & 6.17E-5 & 1.54E-5  \\
order&-&-0.01&0.55&1.73&1.93&1.98& 1.99&2.00&2.00\\ \hline
$\eps_0/2^4$ & 1.78E-2 & 1.76E-2 & 1.80E-2 & \underline{1.82E-2}
& \underline{1.22E-2} & 3.73E-3 & 9.79E-4 & 2.48E-4 & 6.21E-5  \\
order&-& 0.02& -0.03&   -0.02&    0.58&    1.71&    1.93&    1.98 &2.00\\ \hline
$\eps_0/2^5$ & 7.11E-3 & 3.30E-3 & 4.07E-3 & 4.43E-3 & 4.53E-3
& \underline{4.56E-3} & \underline{3.05E-3} & 9.32E-4 & 2.45E-4 \\
order&-&1.11   &-0.30&-0.12&-0.03&-0.01&0.58&1.71&1.93\\ \hline
$\eps_0/2^6$ & 7.19E-3 & 1.99E-3 & 5.10E-4 & 6.84E-4 & 1.02E-3
& 1.10E-3 & 1.13E-3 & \underline{1.14E-3} & \underline{7.61E-4}  \\
order&-&1.85&1.96&-0.42& -0.58& -0.11&-0.04& -0.01&  0.58\\ \hline
$\eps_0/2^7$ & 7.07E-3 & 1.70E-3 & 4.49E-4 & 2.61E-4 & 8.81E-5
& 1.68E-4 & 2.54E-4 & 2.77E-4 & 2.83E-4  \\
order&-&2.06& 1.92&  0.78&   1.57&  -0.93&   -0.60& -0.13&   -0.03\\ \hline
$\eps_0/2^8$ & 7.05E-3 & 1.71E-3 & 4.23E-4 & 1.09E-4 & 3.91E-5
& 6.01E-5 & 2.18E-5 & 4.20E-5 & 6.35E-5  \\
order&-&2.04    &2.02    &1.96    &1.48   &-0.62    &1.46   &-0.95   &-0.60\\ \hline
$\eps_0/2^9$ & 7.05E-3 & 1.71E-3 & 4.22E-4 & 1.05E-4
& 2.61E-5 & 1.37E-5 & 6.98E-6 & 1.50E-5 & 5.48E-6  \\
order&-&2.04&   2.03&    2.01& 2.01&0.93& 0.97&-1.10&   1.45\\
\hline
\bottomrule
\end{tabular}\label{table_temp}
\end{table}

(i) For the spatial discretization error, the MTI-FP method is
uniformly spectral accurate for all $\eps\in(0,1]$ (cf. Tab. \ref{table_spatial}).

(ii) For the temporal discretization error,
the MTI-FP method is uniformly convergent with linear rate at $O(\tau)$ for $\eps\in(0,1]$.
For any fixed $0<\eps\le 1$, when time step $\tau$ is
small, i.e. $\tau\lesssim \eps^2$ (upper triangle part of Tab. \ref{table_temp}), second order
convergence at $O(\tau^2)$ is achieved; when $\eps$ is small, i.e. $\eps\lesssim \tau$
(lower triangle part of Tab. \ref{table_temp}), again second order convergence at $O(\tau^2)$  is achieved.
However, near the diagonal part where $\tau\sim \eps^2$
(cf. the underlined diagonal part of Tab. \ref{table_temp}),  degeneracy of
the convergence rate  and the uniform linear convergence rate for
the temporal error are observed. In particular,
the underlined errors in Tab. \ref{table_temp} degenerate
in the parameter regime $\tau\sim\eps^2$, which has been predicted by our
error estimates \eqref{eq:est:main}.

\subsection{Convergence of the Dirac equation
\eqref{SDEdd} to its limiting models}
\begin{figure}
\centerline{\psfig{figure=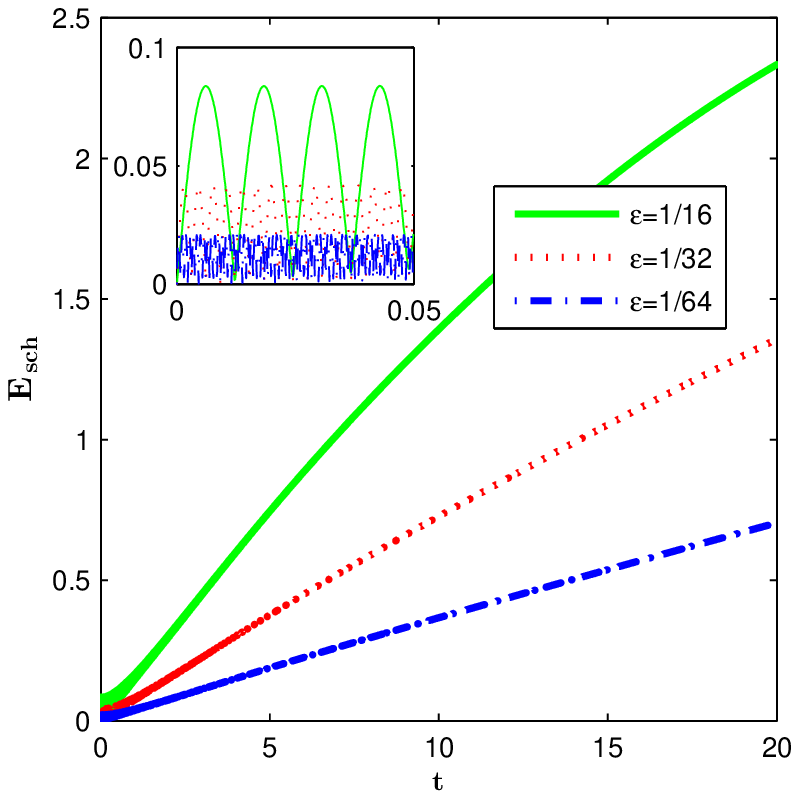,height=6cm,width=7cm}\quad \psfig{figure=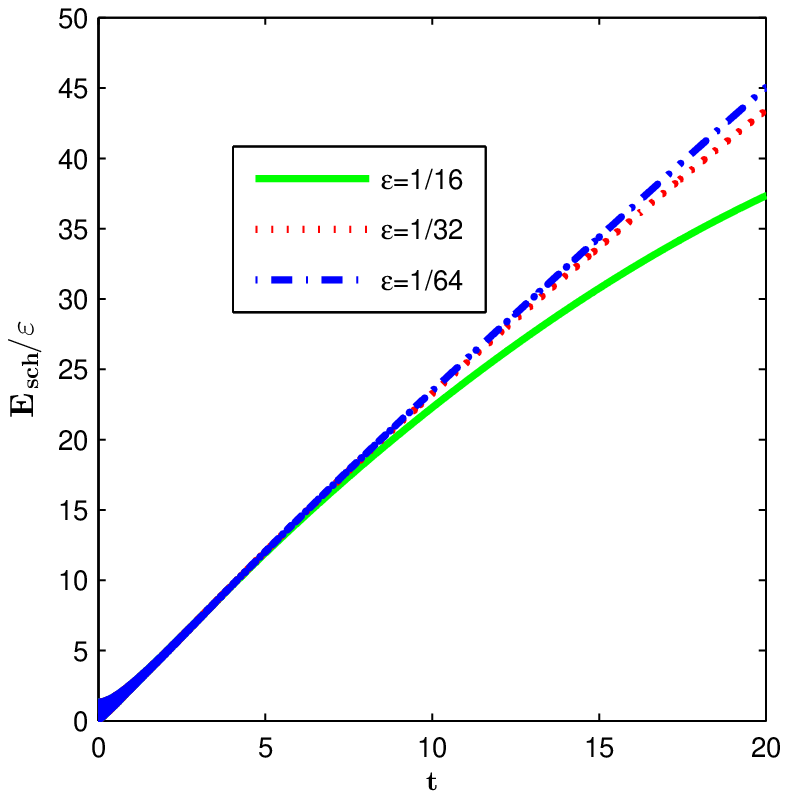,height=6cm,width=7cm}}
\medskip
\centerline{\psfig{figure=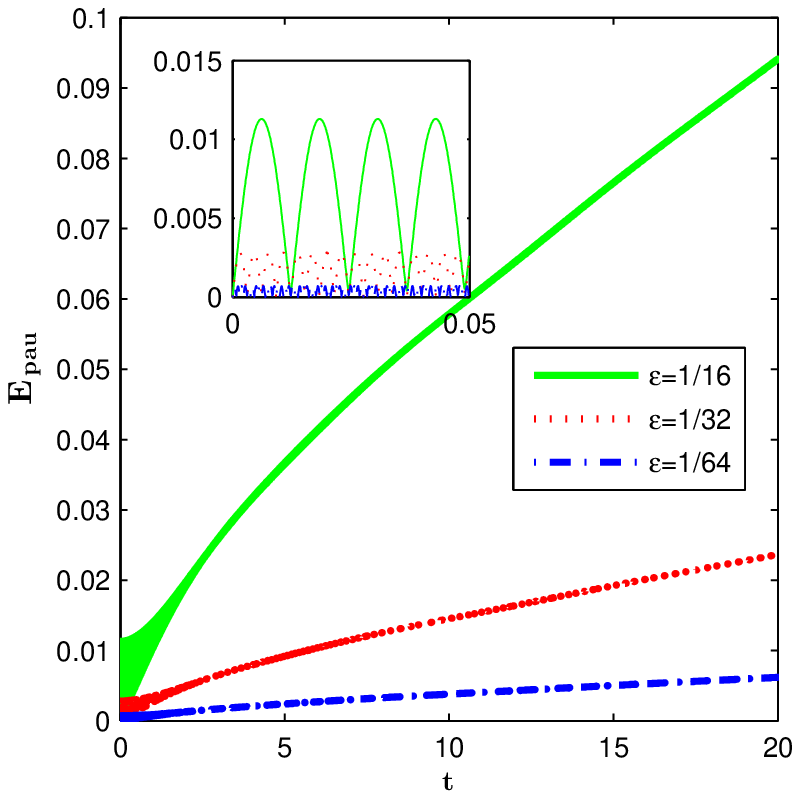,height=6cm,width=7cm}\quad \psfig{figure=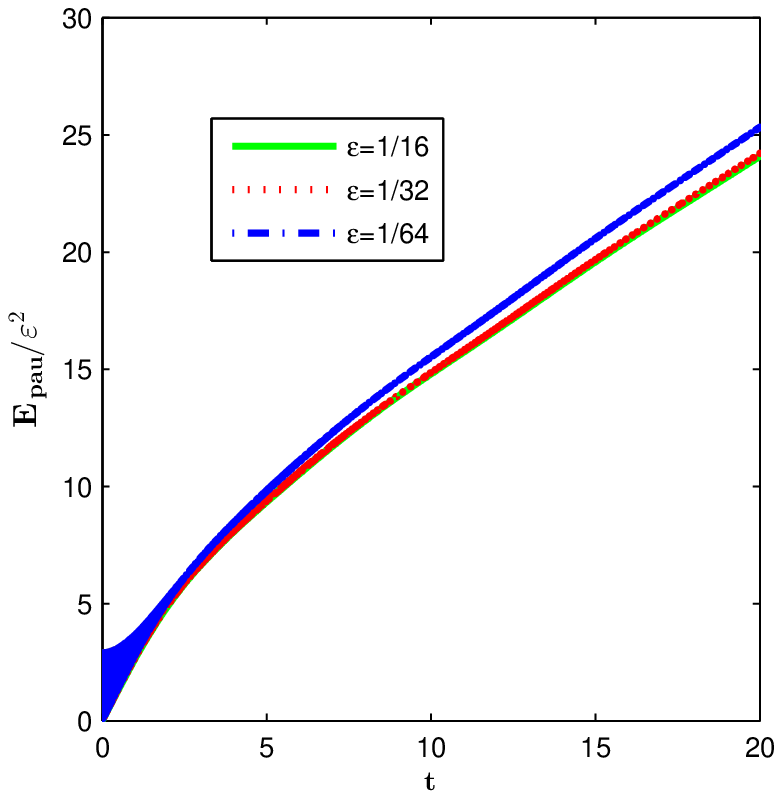,height=6cm,width=7cm}}
\caption{Error functions $E_{\rm sch}(t)$ and $E_{\rm pau}(t)$ for different $\eps$.}\label{fig:0}
\end{figure}
Similarly to \eqref{eq:MD:1st}, when  $\eps\to0^+$, the solution of the Dirac equation \eqref{SDEdd}
 satisfies, i.e. first order limiting model
\be\label{eq:1st}
\Phi(t,x)=e^{-it/\eps^2}\phi_e\,{\bf e}_1+e^{it/\eps^2}
\phi_p\,{\bf e_2}+O(\eps),\quad {\bf e}_1=(1,0)^T,\quad {\bf e}_2=(0,1)^T,
\ee
where $\phi_e:=\phi_e(t,x)\in{\mathbb C}$ and $\phi_p:=\phi_p(t,x)\in{\mathbb C}$
satisfy the Schr\"odinger equations \cite{N,Sch,White,BMP},
\be\label{eq:1st:eq}
i\p_t\phi_e=\left[-\frac{1}{2}\Delta +V(t,x)\right]\phi_e,\qquad i\p_t\phi_p=\left[\frac{1}{2}\Delta+V(t,x)\right]\phi_p,\quad x\in\mathbb{R},\quad t>0,
\ee
and the initial data is determined through \eqref{eq:1st}.

To obtain a second order limiting  Pauli-type equation model,
we formally drop the small components in \eqref{eq:decsys:1}-\eqref{eq:decsys:2} to get
\be\label{eq:2nd}
\Phi(t,x)=e^{-it/\eps^2}\Psi_e(t,x)+e^{it/\eps^2}\Psi_p(t,x)+O(\eps^2),
\ee
where $\Psi_e:=\Psi_e(t,x)\in{\mathbb C}^2$ and
$\Psi_p:=\Psi_p(t,x)\in{\mathbb C}^2$ satisfy the Pauli-type equations
\be\label{eq:2nd:eq}
i\p_t\Psi_e=\frac{1}{\eps^2}\mathcal{D}\Psi_e
+\Pi_+\left(W\Psi_e\right),\quad
i\p_t\Psi_p=-\frac{1}{\eps^2}\mathcal{D}\Psi_p
+\Pi_-\left(W\Psi_p\right),\quad x\in {\mathbb R}, \quad t>0,
\ee
with $\mathcal{D}=\sqrt{I-\eps^2\Delta}-I$ and initial value as
\begin{equation}
\Psi_e(0,x)=\Pi_+\Phi(0,x),\quad \Psi_p(0,x)=\Pi_-\Phi(0,x), \qquad x\in {\mathbb R}.
\end{equation}
To investigate numerically convergence rates of the above
limiting models \eqref{eq:1st} and \eqref{eq:2nd} to the Dirac equation,
we solve numerically the Schr\"odinger equation \eqref{eq:1st:eq} to
obtain $(\phi_e,\phi_p)$ and the Pauli-type equation \eqref{eq:2nd:eq}
to get $(\Psi_e,\Psi_p)$, by the TSFP method \cite{BC1} and the EWI-FP method \cite{BCJ}, respectively.
The solution $\Phi$ of the Dirac equation \eqref{SDEdd}
is computed by the MTI-FP method and we can study convergence rates of Dirac equation
\eqref{SDEdd} to \eqref{eq:1st} and \eqref{eq:2nd}, respectively.
All the computations are done on the bounded interval $\Omega=(-128,128)$
with fine mesh $h=1/16$ and time step $\tau=10^{-4}$, which are sufficiently small such that
the numerical error can be neglected.
In order to quantify the convergence, we introduce the error functions
\begin{equation*}\begin{split}
& E_{\rm sch}(t)=\left\|\Phi(t,\cdot)-e^{-it/\eps^2}
\phi_e(t,\cdot){\bf e}_1-e^{it/\eps^2} \phi_p(t,\cdot){\bf e_2}\right\|_{L^2},\\
& E_{\rm pau}(t)=\left\|\Phi(t,\cdot)-e^{-it/\eps^2}
\Psi_e(t,\cdot)-e^{it/\eps^2} \Psi_p(t,\cdot)\right\|_{L^2}, \quad t\ge0.
\end{split}
\end{equation*}
Fig. \ref{fig:0} depicts the evolution of the errors $E_{\rm sch}(t)$
and $E_{\rm pau}(t)$, and we can conclude that the limiting model of
the Schr\"odinger equation \eqref{eq:1st}
is linearly  accurate at $O(\eps)$, while the limiting model of the
Pauli-type equation \eqref{eq:2nd} is  quadratically  accurate at $O(\eps^2)$.
In particular, both the errors
$E_{\rm sch}(t)$ and $E_{\rm pau}(t)$ are observed to grow linearly in time, i.e.
\be
E_{\rm sch}(t)\leq (C_1+C_2t)\eps,\qquad E_{\rm pau}(t)\leq (C_3+C_4t)\eps^2, \quad t\ge0,
\ee
where $C_1$, $C_2$, $C_3$ and $C_4$ are positive constants independent of time $t\ge0$ and $\eps\in(0,1]$.
We find that \eqref{eq:2nd:eq} is the same second order approximate limiting model
as the Pauli equation \cite{N,NM} for the Dirac equation \eqref{SDEdd}
in the nonrelativistic limit regime.

\section{Conclusion}\label{sec:con}
A multiscale time integrator Fourier pseudospectral (MTI-FP) method
was proposed and rigorously analyzed for the Dirac equation
involving a  dimensionless parameter $\eps\in(0,1]$,
which is inversely proportional to the speed of light.
The main difficulty of the problem is that the solution
highly oscillates with $O(\eps^2)$ wavelength in time when $0<\eps\ll1$.
The key ideas in designing the MTI-FP method included a proper multiscale decomposition
of the Dirac equation and the use of the Gautschi type exponential wave integrator
in time discretization.
Rigorous error analysis showed that the MTI-FP method is uniformly convergent
in spatial  discretization with spectral accuracy,
and uniformly convergent in
temporal discretization with linear order for $\eps\in(0,1]$,
while the temporal accuracy is optimal with quadratic convergence rate
when either   $\eps=O(1)$ or $\eps\lesssim \tau$.
This result significantly improves the
error bounds of the existing numerical methods  for the Dirac
equation in the nonrelativistic limit regime.
Numerical results  confirmed the error estimates and suggested our
error bounds are sharp and optimal. Convergence rates of the Dirac
equation to its limiting first order Schr\"odinger equation model and
second order Pauli-type equation model  were observed numerically.

\end{document}